\documentclass[12pt,reqno]{amsart}
\usepackage{amssymb,amsmath,amsthm,amscd, eucal}
\usepackage{verbatim}
\usepackage{tikz}

\usepackage{color}

\usepackage{enumerate}
\usepackage{bbm}
\usepackage[a4paper]{geometry}
\usepackage{amscd}

\DeclareMathOperator{\Ran}{Ran}

\DeclareMathOperator*{\slim}{s-lim}

\DeclareMathOperator{\clos}{clos}
\DeclareMathOperator{\dist}{dist}

\renewcommand\Re{\text{\rm Re}\,}

\newcommand{\q}{\quad}

\newcommand{\ov}{\overline}
\newcommand{\z}{\zeta}

\newcommand{\e}{\eqref}
 \newcommand{\1}{\mathbbm{1}}
 
\numberwithin{equation}{section}

\theoremstyle{plain}
\newtheorem{theorem}{\bf Theorem}[section]
\newtheorem{lemma}[theorem]{\bf Lemma}

\newtheorem{corollary}[theorem]{\bf Corollary}
\newtheorem{cond}[theorem]{\bf Condition}

\theoremstyle{definition}
\newtheorem{definition}[theorem]{\bf Definition}

\theoremstyle{remark}
\newtheorem*{remark*}{\bf Remark}
\newtheorem{remark}[theorem]{\bf Remark}
\newtheorem{example}[theorem]{\bf Example}

\begin{document}

\title[Toeplitz operators
with piecewise continuous symbols]{Multichannel scattering theory for Toeplitz operators
with piecewise continuous symbols}


\author{Alexander V. Sobolev}
\address{University College London, Department of Mathematics, 
Gower Street, London WC1E 6BT, U.K.}
\email{a.sobolev@ucl.ac.uk}

\author{Dmitri Yafaev}
\address{Univ  Rennes, CNRS, IRMAR-UMR 6625, F-35000
    Rennes, France and SPGU, Univ. Nab. 7/9, Saint Petersburg, 199034 Russia}
\email{yafaev@univ-rennes1.fr}

\date{20 May 2019}

\thanks{Our collaboration has become possible through the hospitality and financial support 
of the Departments of Mathematics  of University College London and of the University of Rennes 1.
The LMS grant is   gratefully acknowledged. The first author was also supported 
by EPSRC grant EP/P024793/1. 
The second author was also supported by the Grant RFBR No. 17-01-00668 A and the Mittag-Leffler Institute}

\begin{abstract} 
Self-adjoint Toeplitz operators have purely absolutely continuous 
spectrum. For Toeplitz operators $T$ with piecewise continuous symbols, 
we suggest a further spectral classification determined by  propagation properties of the 
operator $T$, that is, by the behavior of $\exp(-iTt) f$ for $t\to\pm\infty$. It turns out that the spectrum is naturally partitioned into three disjoint subsets: \textit{thick}, \textit{thin} 
and \textit{mixed} spectra. On the thick spectrum,  the propagation properties are modelled by the 
continuous part of the symbol, whereas on the thin spectrum, the model operator is determined by the 
jumps of the symbol. On the mixed spectrum, these two types of the asymptotic evolution of $\exp(-iTt) f$ coexist. 
This classification is justified in the framework of scattering theory.
We prove the existence of wave operators that relate the model operators with the Toeplitz operator $T$.
The ranges  of these wave operators are pairwise orthogonal, and their orthogonal sum exhausts the whole space, i.e., the set of these wave operators is asymptotically complete. 
\end{abstract}



\maketitle

\section{Introduction}

\subsection{Basic notions}
The Toeplitz operator $T = T(\omega)$ with symbol $\omega (\z)$ 
is defined on the Hardy space $\mathbb H^2 = \mathbb H^2_+$ 
on the unit circle $\mathbb T = \{\z\in {\mathbb C}: |\z|=1\}$ by the formula 
 \begin{equation}
(T u)(\z) = (\mathbb P\omega u)(\z), \quad u\in \mathbb H^2,\ \z\in\mathbb T, 
\label{eq:TT}\end{equation}
 where $\mathbb P = \mathbb P_+: L^2(\mathbb T)\to \mathbb H^2$ is the orthogonal projection 
 onto $\mathbb H^2$. 
The normalized Lebesgue measure on $\mathbb T$ is denoted 
by $d{\mathbf m}(\z) = (2\pi i \z)^{-1} d\z$.
Throughout the paper we assume that $\omega\in L^\infty(\mathbb T)$, 
so that the operator $T$ is   bounded. 
If $\omega$ is real-valued, then the operator $T$ is self-adjoint. 
 
In the early 60's M. Rosenblum (see \cite{Ros1, Ros2, Ros3}) established a number of fundamental spectral 
properties of
self-adjoint Toeplitz operators $T$. Namely 
if $\omega$ is a non-constant function,  then $T$  
is absolutely continuous and its spectrum $\sigma(T)$ 
fills the interval 
$[\gamma_1, \gamma_2]$ where 
\begin{equation}
\gamma_1 = \textup{ess}\inf\omega, \q\gamma_2 = \textup{ess}\sup\omega. 
\label{eq:gamma}\end{equation}
Furthermore, M. Rosenblum (see also the paper \cite{Ismag} by R.~S.~Ismagilov) has constructed a spectral representation of $T$ and described its spectral multiplicity.

In this paper we consider Toeplitz operators with  piecewise continuous symbols $\omega$.
It is natural to compare $T$ with the operator $\boldsymbol\Omega$ of 
multiplication 
by a piecewise continuous function $\omega$ on $L^2(\mathbb T)$. 
The spectrum $\sigma(\boldsymbol\Omega)$ is the closure 
$\clos\omega(\mathbb T)$ of the set $\omega(\mathbb T)$. 
A point $\lambda\in \sigma(\boldsymbol\Omega)$ is an eigenvalue of $\boldsymbol\Omega$ if and only if 
the  measure 
${\mathbf m}\big(\{\z\in \mathbb T: \omega(\z) =~\lambda\}\big)$ is positive. 
If $\omega$  is piecewise $C^1$, then, 
apart from eigenvalues, the spectrum of $\boldsymbol\Omega$ is  absolutely continuous, see
 Lemma~\ref{aacc}.
Also, the multiplicity of the spectrum is found as the number of solutions of the equation 
$\omega(\z) = \lambda$. In particular, the eigenvalues of $\boldsymbol\Omega$
can have only infinite multiplicity.

Our main aim is to construct scattering theory for Toeplitz operators 
with piecewise continuous symbols. We start, however, with the case of smooth symbols and 
prove the existence, isometry and completeness of the wave operators
$W_{\pm} (T(\omega), \boldsymbol\Omega; {\mathbb P})$ for the pair 
$\boldsymbol\Omega, T(\omega)$ and the ``identification" $\mathbb P  : L^2(\mathbb T)\to \mathbb H^2$. 
In this case $\sigma(T)=\sigma(\boldsymbol\Omega)$.

For general piecewise continuous symbols, 
  $\sigma(\boldsymbol\Omega)$ may have gaps, so 
only the inclusion $\sigma(\boldsymbol\Omega)\subset [\gamma_{1}, \gamma_{2}]=\sigma(T)$ holds.
At first glance, this looks counter-intuitive since $T$ is a compression of $\boldsymbol\Omega$. 
 The set 
$\sigma(T)$ splits into two  disjoint subsets $\sigma(\boldsymbol\Omega)$ 
and $\sigma(T)\setminus \sigma(\boldsymbol\Omega)$. 
It turns out that the spectral nature
of these two components is qualitatively different.
In some sense, the spectrum of the Toeplitz operator $T$ is \textit{thin} on $\sigma(T)\setminus \sigma(\boldsymbol\Omega)$, and it is
\textit{thick} (or \textit{mixed}) on the set $\sigma(\boldsymbol\Omega)$.
 We give  precise definitions of these  
terms and prove corresponding results in the scattering theory framework, studying the asymptotic behaviour 
of the time evolution $e^{-iTt}f$ as $t\to\pm\infty$ for various $f\in \mathbb H^2$.


\subsection{Truncated Toeplitz matrices} 

Previously, Toeplitz operators with discontinuous symbols 
were extensively studied in a 
different context: instead of the operator $T$ one considered \textit{truncated} Toeplitz matrices 
$T_n = T_n(\omega)$ as $n\to\infty$.
Here $T_n$ are defined as follows. 
Let $\mathcal P_n\subset {\mathbb H}^{2} $ be the subspace of all polynomials 
of degree $\leq n$, and let ${\sf P }_{n}$ be the orthogonal projection onto $\mathcal P_n$. 
The operator $T_{n}$ acting on the space $\mathcal P_{n}$ is defined by the formula $T_{n}f = {\sf P }_n T f$. 
The difference between the two components of $\sigma(T)$ 
becomes clearly visible when one studies  
the asymptotics of the eigenvalue 
counting function $N_n(a, b)$ for the matrix $T_n$ on an interval $(a, b)$ as $n\to\infty$. 
It is straightforward to deduce  from the classical result in  
\cite[Sect 5.2]{Gren_Szego_1958} the following formula.  
If 
$\mathbf m\big(\{\z: \omega(\z) = a\}\big) 
= \mathbf m\big(\{\z: \omega(\z) = b\}\big) = 0$, then 
\begin{align*}
\lim\limits_{n\to\infty}n^{-1} N_n(a, b) = \mathbf m\big(\{\z: a < \omega(\z) < b\}\big). 
\end{align*}
In particular, if $(a, b)\subset \sigma(T)\setminus\sigma(\boldsymbol\Omega)$,   
then the right-hand side is zero. 
This is consistent with our choice of the term ``thin spectrum" for the set 
$\sigma(T)\setminus\sigma(\boldsymbol\Omega)$.

Starting with the pioneering Szeg\H o's paper 
\cite{Szego1915}, a 
substantial body of research was focused on the asymptotic behaviour of the determinants 
of the truncated Toeplitz matrices 
$T_n = T_n(\omega)$, as $n\to\infty$, with complex-valued $\omega$. 
The case of discontinuous symbols $\omega$ attracted a special attention in connection 
with the so-called Fischer-Hartwig formula, see  \cite{Fisher1968} 
and   \cite{DIK2011, Kras}. 

\subsection{Scattering theory framework} 

For smooth symbols $\omega$, 
the asymptotic behaviour of the evolution $e^{-iTt}f$ as $t\to\pm\infty$ 
is described by the same formula for all $f\in \mathbb H^2$. Precisely,   
for every $f$ there exists an $f_\pm\in L^2 (\mathbb T)$ such that
\begin{align}\label{evo:eq}
e^{-iTt} f \sim \mathbb P e^{-i\boldsymbol\Omega t}f_\pm,\quad t\to \pm\infty,
\end{align}
where the symbol ``$\sim$" means that the difference of the left- and right-hand sides 
tends to zero. 

To describe the evolution $e^{-iTt} f $ for discontinuous symbols $\omega$,  we distinguish three 
components of the spectrum $\sigma(T)$. Assume that $\omega$ is continuous on $\mathbb T$ apart from 
some finite subset ${\sf S}\subset\mathbb T$. For each $\eta\in {\sf S}$, we set
\begin{align*}
a_- = \min\{\omega(\eta-0), \omega(\eta+0)\}, \q
a_+ = \max\{\omega(\eta-0), \omega(\eta+0)\}, 
\end{align*}
and  define 
\begin{align}
\Lambda_\eta = [
a_-, 
a_+], \q \Upsilon = \bigcup_{\eta\in  {\sf S}} \Lambda_\eta.
\label{eq:Jump2}
\end{align} 
Clearly,
\begin{align*}
\sigma(T)= \sigma(\boldsymbol\Omega )\cup  \Upsilon  \q \mbox{and}\q
\sigma(T)\setminus\sigma(\boldsymbol\Omega )\subset   \Upsilon \subset \sigma(T)   .
\end{align*}
Denoting by $E_T(X)$, $X\subset \mathbb R$,  the spectral projection of the operator $T$, 
introduce the pair-wise orthogonal subspaces 
\begin{align}\label{decompose:eq}
\left. 
\begin{aligned}
\mathcal H_{\rm thin} = &\ \Ran E_T\big(\sigma(T)\setminus\sigma(\boldsymbol\Omega)\big),\\[0.2cm]
\mathcal H_{\rm thick} = &\ \Ran E_T\big(\sigma(\boldsymbol\Omega)\setminus  \Upsilon  \big),\\[0.2cm]
\mathcal H_{\rm mix} = &\ \Ran E_T\big( \sigma(\boldsymbol\Omega)\cap  \Upsilon  \big),
\end{aligned}
\right\}
\end{align}
so that
 \[
\mathcal H_{\rm thin}\oplus \mathcal H_{\rm thick}\oplus\mathcal H_{\rm mix} = \mathbb H^2. 
\]

\begin{definition}\label{class}
The spectrum of the operator $T$ on the subspaces 
$\mathcal H_{\rm thin}$, $\mathcal H_{\rm thick}$ and $\mathcal H_{\rm mix}$
is said to be thin, thick and mixed respectively. 
\end{definition}

  Up to   sets of measure zero, these three components of the spectrum are disjoint. This yields a fine classification of the absolutely continuous spectra for Toeplitz operators with piecewise continuous symbols.

If the symbol $\omega$ is continuous, then $ \Upsilon  =\varnothing$ so that $\sigma(T)= \sigma(\boldsymbol\Omega)  $. In this case $\mathcal H_{\rm thick} = \mathbb H^2$ and the
whole spectrum is thick. On the other hand, if $\omega$ is piecewise constant (but not a constant function), 
then the Lebesgue measure $|\sigma(\boldsymbol\Omega)|=0$ and $\mathcal H_{\rm thin}=\Ran E_T\big(\sigma(T) \big) = \mathbb H^2$, i.e., the whole spectrum is thin. 

  Let us illustrate this partition of the spectrum 
using the example of symbol $\omega$ in Fig~1. 
According to Theorem \ref{Ros_smooth:thm}, the spectrum of $T(\omega)$ is simple.  
In this case the multiplicity of the spectrum of $ \boldsymbol\Omega$ does not exceed $2$, see 
Theorem \ref{omega}.
If for a   spectral interval $\Lambda\subset\sigma (\boldsymbol\Omega)$ 
it equals $2$, then the spectrum of $T$ on $\Lambda$ is thick.
If it equals $1$, then the spectrum of $T$ on $\Lambda$ is mixed. 
 The gap in the spectrum of $\boldsymbol\Omega$ produces the thin spectrum
of $T$.

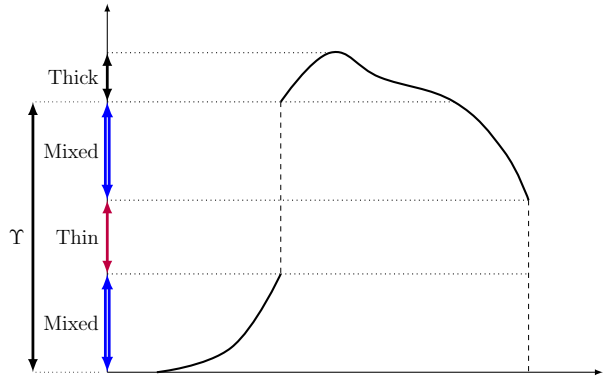
\begin{figure}
\resizebox{8cm}
{!}{
\begin{tikzpicture}\draw [very thick] plot[smooth, tension=.7] coordinates {(-4,-4) (-2.5,-3.5) (-1.5,-2)};
\draw [very thick] plot[smooth, tension=.7] coordinates {(-1.5,1.5) 
(-0.5,2.5) (0.5,2) (2,1.5) (3,0.5) (3.5,-0.5)};
\draw [dashed](-1.5,1.5) node (v2) {} -- (-1.5,-2) node (v3) {};
\draw [dashed](3.5,-0.5) node (v4) {} -- (3.5,-4);
\draw [semithick][-latex](-5,-4) node (v1) {} -- (5,-4);
\draw [semithick][-latex](-5,-4)  -- (-5,3.5);
\draw [semithick][dotted](-5,1.5) node (v8) {} -- (-1.5,1.5);
\draw [semithick][dotted](-5,-2) node (v5) {} -- (-1.5,-2);
\draw [semithick][dotted](2,1.5) -- (-1.5,1.5);
\draw [semithick][dotted](3.5,-0.5) -- (-5,-0.5) node (v7){};
\draw [semithick][dotted](-0.5,2.5) -- (-5,2.5) node (v9) {};
\draw [semithick][dotted] (-1.5,-2) -- (3.5,-2);
\draw [double][semithick][ultra thick, blue][latex-latex](-5,-4) node (v10) {} --(-5,-2)
node [pos=.5, left, black]  {Mixed};
\draw [semithick][ultra thick, purple][latex-latex](-5,-2) --(-5,-0.5)
node [pos=.5, left, black]  {Thin};
\draw [double][ultra thick, blue][latex-latex](-5,-0.5) -- (-5,1.5)
node [pos=.5, left, black]  {Mixed};
\draw [semithick][ultra thick, black][latex-latex](-5,1.5) node (v6) {} -- (-5,2.5)
node [pos=.5, left, black]  {Thick};
\draw[dotted] (v6) -- (-6.5,1.5) node (v11) {};
\draw[dotted] (v10) -- (-6.5,-4) node (v12) {};
\draw [semithick][ultra thick, black][latex-latex] (-6.5,1.5) -- (-6.5,-4) 
node [pos=.5, left, black]  {$ \Upsilon  $};
\end{tikzpicture}}
\caption{Partition into thick, thin and mixed spectra}
\end{figure}

These definitions are justified in terms of the asymptotic behaviour of $e^{-iTt} f$ as 
$t\to\pm \infty$ for $f$ from each of these subspaces. 
If $f\in \mathcal H_{\rm thick}$, then the asymptotic 
formula \eqref{evo:eq} holds. 
On the subspace $\mathcal H_{\rm thin}$, the evolution is modelled by 
the Toeplitz operators $T_{\rm s} = \omega_{\rm s}$ with symbols $\omega_{\rm s}$  
that are step functions on $\mathbb T$. 
Here we use  that for symbols which are indicators of arcs, 
the operator 
$e^{-iT_{\rm s}t} $ admits an  explicit representation, 
see Sect. \ref{jumps:sect}. On the subspace $\mathcal H_{\rm mix}$, the evolution $e^{-iTt}f$ is, asymptotically, 
a linear combination of evolutions for the thin and thick cases.

A surprising fact is that, for $f\in \mathcal H_{\rm mix}$,  the asymptotic
behaviour of $e^{-iTt}f $ as $t\to\pm\infty$  qualitatively depends  
on the sign of $t$. For instance, 
the evolution may be  ``thick" for one of the signs  and ``thin"  for the other sign. 
This phenomenon is discussed in Section~5.4.

The difference between thick and thin spectra  
manifests itself also  in the structure of the continuous spectrum 
eigenfunctions. 
As discussed in Sect. \ref{diag:subsect}, 
an arbitrary Toeplitz operator with a simple spectrum can be reduced to the multiplication operator by $\lambda$ in the space $L^2 (\gamma_{1}, \gamma_{2})$ 
by a unitary operator $\Phi  : {\mathbb H}^2\to L^2 (\gamma_{1}, \gamma_{2})$ defined by a formal relation
 \begin{equation}
(\Phi f)(\lambda) = \int_{\mathbb T}\ov{\varphi(\z;\lambda) } f(\z) d{\bf m}(\z).
\label{eq:Ros3}
\end{equation}
This means that for any $f\in \mathbb H^2$, 
\begin{equation}
(\Phi T f)(\lambda)=\lambda (\Phi  f)(\lambda), \q {\rm a.e.}\q \lambda\in (\gamma_1, \gamma_2).
\label{eq:Ros4}
\end{equation}
The functions $\varphi(\z;\lambda)$ satisfy the equation 
$T(\omega) \varphi(\lambda)= \lambda \varphi(\lambda)$, 
in an appropriate generalized sense (see \cite[Sect. 5]{SobYaf} 
for details of this construction), 
and hence can be interpreted as eigenfunctions of the operator $T(\omega)$.
We note that

--  eigenfunctions of  the thick spectrum  
  have stronger singularities
   than those of the thin one,

--   on the thick spectrum, the position of singularities depends on the spectral parameter $\lambda$ while
on the thin spectrum, singularities are located at the points of discontinuity irrespectively of the value of $\lambda$. 

In Section \ref{examples:subsect} we consider two   examples where  eigenfunctions  are given by explicit formulas: one 
with a smooth symbol, and the 
other one with a symbol which is the indicator of an arc of $\mathbb T$. 
 For brevity we call such symbols \textit{indicator type} symbols. 
In the first case the entire spectrum is thick, and in the second case the entire  spectrum is thin. 

 \subsection{Analytic background} 

We study the evolution $e^{-i Tt} f$ by investigating 
wave operators for the pairs $\boldsymbol\Omega, T$ and $T_{\rm s}, T$. 
 Our main results are the existence and completeness of an appropriate set of
wave operators.

From the analytic point of view, our approach comprises three ingredients:

\begin{enumerate}[\rm (i)]

\item 
To investigate the thick spectrum, we use 
trace class scattering theory which allows us to prove the existence of  the wave 
operators for the pair $\boldsymbol\Omega, T$.   
In the case of smooth symbols, the trace class framework yields also the completeness of 
 these wave  operators.
 
\item
For the thin spectrum, we rely on Cook's criterion 
which leads to the existence of wave operators for the 
pairs $T_{\rm s}, T$ with   model Toeplitz
operators $T_{\rm s}=T (\omega_{s})$ whose symbols 
$\omega_{s}$ are appropriately chosen step functions. 

\item
In contrast to the smooth case, where 
there is only one scattering channel described by formula \eqref{evo:eq}, 
for symbols with discontinuities each jump creates a new scattering channel. 
Thus our final task is to show \textit{asymptotic completeness}, that is to prove that the constructed 
scattering channels exhaust the whole space. 
Our proof of the asymptotic completeness     
relies on the  results on spectral multiplicity of Toeplitz
operators. 
\end{enumerate}

\subsection{Related problems} 
We mention two such problems. 
The first one is the three-particle (and more generally, a multi-particle) 
quantum-mechanical problem. 
In this problem, it also makes sense to classify 
bands of the spectrum as being ``thick" or ``thin", although 
the precise meaning of these terms is, of course, different from the one 
introduced above for Toeplitz operators. 
The  ``thick'' spectrum describes scattering states where all 
particles are asymptotically free, 
while the  ``thin'' spectrum corresponds to scattering of one of 
the particles on a bound state of the other two. 
The ``thin'' spectrum has essentially two-particle nature. 
 
The second problem is scattering theory for Hankel operators with discontinuous symbols.  
While Hankel operators with smooth symbols are compact, they acquire  
an absolutely continuous spectrum if their symbols have jumps. 
This phenomenon has been 
known in concrete examples, such as the Carleman operator, 
see \cite[Chapter 10]{Pe}. 
A general spectral theory of Hankel operators with discontinuous 
symbols was built in \cite{PuYaf}.  
 Interestingly, continuous spectrum   eigenfunctions 
of Hankel (see \cite[Sect. 4]{PuYaf}) and Toeplitz operators with discontinuous symbols
  have singularities of the same type.

\subsection{Structure of the paper}

The paper is organized as follows. 
In Sect. 2, we collect known 
facts about Toeplitz operators necessary for the rest of the paper. In particular, we describe in Theorem 
\ref{Ros_smooth:thm} the spectral multiplicity of operators with  piecewise continuous symbols $\omega$ in terms adapted to our applications to scattering theory.
In Sect. 3, 
we provide some general information about the trace class scattering theory.  Then we
investigate wave operators for the pair $ \boldsymbol\Omega, T$. This is already sufficient for construction of scattering theory for  Toeplitz operators with smooth symbols (see Theorem~\ref{deltapm:thm}).
Sect. 4 focuses on wave operators for symbols with jump discontinuities. 
This requires an analysis of the propagator of the model Toeplitz operator, 
i.e.,  operator with an indicator type symbol. 
  In Sect. 5,  we prove the  
asymptotic completeness of the set of wave operators constructed in the previous sections. The main result is stated as 
Theorem~\ref{Main}.

 \subsection{ Notation}


The symbol $\mathbb D$ is used to denote the unit disk $\{z\in\mathbb C: |z|<1\}$. 
On $\mathbb T$ we define the 
standard measure $d{\bf m}(\z) =   (2\pi i\z)^{-1} d\z$, $\z\in\mathbb T$.  
For any $\z_1, \z_2\in \mathbb T$, we denote by 
   $(\z_1, \z_2)$   the   open arc  
joining $\z_1$ and $\z_2$ 
counterclockwise. 
For a function $\omega$ defined on $\mathbb T$, 
the notation $\omega(\z\pm 0)$ stands for the limit $\lim_{\varepsilon\downarrow 0}
\omega(\z e^{\pm i\varepsilon})$, if it exists. 
Let $\1_{\Delta}$ be the indicator of a set $\Delta\subset\mathbb T$; 
the operator in $L^{2}({\mathbb T})$ of multiplication by this 
function is also denoted by $\1_{\Delta}$. 

Denote by $F: L^1 ({\mathbb T})\to \ell^\infty ({\mathbb Z})$ 
the discrete Fourier transform: 
\begin{align}\label{fourier:eq}
(Fu)_n = \int_{\mathbb T} \z^{-n} 
u(\z) d\mathbf m(\z),\q n\in \mathbb Z, \; u\in L^2(\mathbb T).
\end{align}   
Considered as a mapping of $L^2(\mathbb T)$ onto 
$\ell^2(\mathbb Z)$, the operator $F$ is unitary. 
The notation $\mathbb H^2_{\pm} = \mathbb H^2_{\pm}(\mathbb T)$ 
stands for the  Hardy classes 
\begin{align*}
\mathbb H^2_+ = &\q \{ f\in L^2(\mathbb T): (Ff)_n  = 0,\q n <0\},\\
\mathbb H^2_- = &\q \{ f\in L^2(\mathbb T): (Ff)_n  = 0,\q n \ge 0\}.
\end{align*}  
As a rule, we set $\mathbb H^2 = \mathbb H^2_+$.  
By $\mathbb P_\pm$ we denote the orthogonal projection from 
$L^2 = \mathbb H^2_+\oplus\mathbb H^2_-$ onto $\mathbb H^2_\pm$.  
Usually we 
use the standard realization   of $\mathbb H^2$ as the space of functions $u$ that are 
analytic on the disk $\mathbb D$ and such that the $L^2(\mathbb T)$-norms 
$\|u(r\ \cdot\ )\|$ are bounded uniformly 
in $r<1$; see, for example, the book \cite{Hof}.

It is often convenient to identify $\mathbb T$ with 
$\mathbb R/ (2\pi\mathbb Z)$ by writing $\z = e^{ix}$ for each $\z\in \mathbb T$. 
Thus with every arc $\Delta\subset \mathbb T$ we associate an interval $X\subset \mathbb R/ (2\pi\mathbb Z)$, and write 
$\Delta = e^{iX}$. Under this change of variables we have $d\mathbf m(\z) = (2\pi)^{-1}dx$. 
Throughout the paper we use the notation  $w(x) = \omega(\z)$.
It is also convenient to set $\omega' (\z):= w'(x)$.

Let $A$ be a self-adjoint operator   on a Hilbert space $\mathcal H$.
By $E_A(\Lambda) $ we denote its spectral projection 
   associated with a Borel set $\Lambda\subset \mathbb R$; $\mathcal{ H}^{\rm(ac)}_{A} $ is the absolutely continuous subspace of $A$; 
$P^{\rm(ac)}_{A} $ is the orthogonal projection on   $\mathcal{ H}^{\rm(ac)}_{A} $ and $A^{\rm(ac)}$ is the restriction of the operator $A$ on $\mathcal{ H}^{\rm(ac)}_{A} $.

Let an operator $A$ and an interval $\Lambda\subset{\mathbb R}$ 
be such that the operator $E_{A}(\Lambda) A $ is absolutely continuous. 
Suppose that there exists a Hilbert space $\mathfrak N$ and 
a unitary mapping $U: E_{A}(\Lambda)  \mathcal{H}\to L^{2} ( \Lambda; \mathfrak{N} )$ 
such that the operator $U E_{A}  (\Lambda) A U^{*}$ acts as multiplication by independent variable $\lambda$ on the space 
$L^{2} (\Lambda; \mathfrak{N} )$. 
Then we say that the spectral representation of $A E_A(\Lambda)$ is realized on  
$L^2(\Lambda; \mathfrak N)$.  We denote by $m_{A}(\Lambda):= \dim \mathfrak N$ the multiplicity of the spectrum of $A$ on $\Lambda$.
 
By $C, c$ (with or without indices) we denote various positive constants whose precise values
are of no importance. 

Unless   stated otherwise,
throughout the whole paper we assume that $\omega$ is a real-valued non-constant function and that 
$\omega\in L^\infty(\mathbb T)$. 
 
\section{Spectral analysis of Toeplitz and  multiplication operators}

 In this section we study spectra of Toeplitz $T$ and multiplication $\boldsymbol\Omega$ operators for piecewise continuous symbols $\omega$.

\subsection{Multiplication operators}

Spectral analysis of the multiplication operator is elementary. It is based on the 
simple relation $E_{\boldsymbol\Omega}(\Lambda) = \1_{\omega^{-1} (\Lambda)}$ 
valid for any Borel set $\Lambda\subset \mathbb R$. 
The next two lemmas are straightforward.
%

\begin{lemma}\label{Omegap:lem}  
Let $\Delta\subset\mathbb T$ be an open set.
Suppose that $\omega \in C^1(\Delta)$ and $\omega'(\z)\not = 0$ for $\z\in\Delta$.
Then $L^2(\Delta)\subset \mathcal{H}^{\rm (ac)}_{\boldsymbol\Omega}$ 
so that the 
 spectrum of the  
operator $\boldsymbol\Omega$ restricted to the subspace $L^2(\Delta)$ is absolutely continuous.
Moreover,  if $\Delta$ is an arc, then   this spectrum 
is simple. 
\end{lemma}
 
\begin{proof}
Let $\Delta$ be an arc, and let $\Lambda = \omega (\Delta)$. Define
the unitary operator $U: L^2(\Delta)\to L^2(\Lambda)$  by the formula
\begin{align*}
(Uf)(\lambda) = \frac{1}{\sqrt{|\omega'(\omega^{-1}(\lambda))|}}
f((\omega^{-1}(\lambda)),\q \lambda\in \Lambda.
\end{align*} 
Since $(U\boldsymbol\Omega f)(\lambda) = \lambda (Uf)(\lambda)$ for $f\in L^2(\Delta)$, 
the restriction of the operator $\boldsymbol\Omega$ to $L^2(\Delta)$ is unitarily equivalent 
to the operator of multiplication by $\lambda$ on $L^2(\Lambda)$.  
This concludes the proof for arcs $\Delta$. In the general case, we only have to apply the result obtained to every constituent arc of $\Delta$.
\end{proof}

\begin{lemma}\label{aacc}
Let $\Delta\subset\mathbb T$ be an open set of full measure.   
Suppose that $\omega\in C^1(\Delta)$. Let 
\begin{align}\label{eq:dpm}
\Delta^{(\pm)} = \{\z\in \Delta: \pm \omega'(\z) < 0\},\quad 
\Delta_{0} =   \{\z\in \Delta: \omega'(\z) = 0\}.
\end{align}  
Then 
\begin{equation}
P^{(\rm ac)}_{\boldsymbol\Omega} = \1_{\Delta\setminus \Delta_{0}}= \1_{\Delta^{(+)}} + \1_{\Delta^{(-)}} .
\label{eq:ac1}\end{equation}
\end{lemma}

\begin{proof}  
Since ${\bf m}({\mathbb T}\setminus\Delta)=0$ and 
$\Delta=\Delta_{0}\cup \Delta^{(+)}\cup \Delta^{(-)}$ is a union of disjoint sets, we have
\begin{equation}
\1=\1_{\Delta}
=  \1_{\Delta_{0}} +  \1_{\Delta^{(+)}} + \1_{\Delta^{(-)}}.
\label{eq:ac}\end{equation}
It follows from Lemma~\ref{Omegap:lem} that   
\begin{equation}
  \1_{\Delta^{(\pm)}}P^{(\rm ac)}_{\boldsymbol\Omega}=   \1_{\Delta^{(\pm)}}.
  \label{eq:acy}\end{equation}
By Sard's theorem, $|\omega (\Delta_{0})|=0$ whence 
$E_{\boldsymbol\Omega}(\omega(\Delta_0)) P_{\boldsymbol\Omega}^{(\rm ac)} = 0$. 
Since $$\1_{\Delta_0}\leq \1_{\omega^{-1}(\omega (\Delta_0))} =E_{\boldsymbol\Omega}(\omega(\Delta_0)),$$ we infer that 
\begin{equation}\label{eq:acx}
\1_{\Delta_{0}} P^{(\rm ac)}_{\boldsymbol\Omega}= 0.
\end{equation}
Putting together \e{eq:ac} and \e{eq:acy}, \e{eq:acx}, we conclude the proof of \e{eq:ac1}.
\end{proof} 

Further results can be obtained under more specific assumptions on the symbol $\omega$. 

\subsection{Piecewise continuous symbols}
From now on we assume that the symbol $\omega$ is piecewise continuous in the following sense.

\begin{cond} \label{SM:cond}
  \begin{enumerate}[\rm(i)]
\item 
$\omega = \overline{\omega}\in L^\infty(\mathbb T)$ and $\omega$ is not a constant function, 
\item 
There exists a finite set ${\sf S} = \{\eta_k\}\subset\mathbb T$ 
such that $\omega\in C^1(\mathbb T\setminus {\sf S})$, 
\item 
The limits $\omega(\eta_k + 0)$, $\omega(\eta_k-0)$ 
exist for all $\eta_k\in {\sf S}$. 
For every $\eta_k\in {\sf S}$, either $\omega(\eta_k+0)\neq  \omega(\eta_k-0)$ or
  $\omega(\eta_k+0)= \omega(\eta_k-0)$ but
the derivative 
$\omega'$ is not continuous at $\eta_k$.
\end{enumerate}
\end{cond}

 Let the sets $\Delta^{(\pm)}$ and the \textit{critical set} $\Delta_0$ be as defined in 
\eqref{eq:dpm} with $\Delta = \mathbb T\setminus {\sf S}$. 
The image $\Lambda_{\rm cr}=\omega (\Delta_0)$ consists of critical values of 
$\omega$. As already observed earlier, 
by Sard's theorem, the Lebesgue measure $|\Lambda_{\rm cr}|=0$.       
We introduce also the ``threshold" set $\Lambda_{\rm thr}$ which consists of all values 
$\omega(\eta_k \pm 0)$, $\eta_k\in {\sf S}$,   
and define the exceptional set 
\begin{align*}
\Lambda_{\rm exc} =\Lambda_{\rm cr} \cup \Lambda_{\rm thr}.
\end{align*}
Since the set  $\Lambda_{\rm thr}$ is finite, $|\Lambda_{\rm exc}| = 0$. 
Moreover, the set  $\Lambda_{\rm exc}$ is closed (see \cite[Lemma 6.2]{SobYaf}).

\textit{Let us fix an interval $\Lambda = (\lambda_{1},   \lambda_{2})$ such that}
\begin{equation} 
\Lambda \subset (\gamma_{1}, \gamma_{2})\setminus  \Lambda_{\rm exc},
\label{eq:TH}
\end{equation} 
where $\gamma_1, \gamma_2$ are defined in \eqref{eq:gamma}. 
Note that $\omega'(\z)\neq 0$ on $\omega^{-1}(\Lambda)\subset {\mathbb T}\setminus {\sf S}$.
Consider   the open set 
\begin{equation}
\delta^{(\pm)}=\delta^{(\pm)}(\Lambda)= \Delta^{(\pm)}\cap\omega^{-1}(\Lambda)\q \mbox{where} \q \omega : {\mathbb T}\setminus {\sf S}\to {\mathbb R}.
\label{eq:AUU}
\end{equation} 
The open set $\delta^{(\pm)}\subset {\mathbb T}\setminus {\sf S}$ is a union of some number, denoted $n^{(\pm)} $, of disjoint open arcs  
$\delta^{(\pm)}_k =(\alpha^{(\pm)}_k, \beta^{(\pm)}_k)$ such that  
$\pm \omega'(\z) < 0$ for  $\z\in\delta^{(\pm)}_k $. Observe that $\omega (\alpha^{(+)}_k)=\lambda_{2}$, $\omega (\beta^{(+)}_k)=\lambda_{1}$ and $\omega (\alpha^{(-)}_k)=\lambda_{1}$, $\omega (\beta^{(-)}_k)=\lambda_{2}$.
In particular, we see that $\omega (\delta^{(\pm)}_k ) = \Lambda$;  
 see Fig. 2 for illustration. It follows that
\begin{equation}
\omega^{-1}(\Lambda)= \delta^{(+)} \cup \delta^{(-)} 
= \Big( \bigcup_{k=1}^{n^{(+)}} \delta^{(+)}_k \Big)
\cup \Big(\bigcup_{k=1}^{n^{(-)}} \delta^{(-)}_k \Big)
\label{eq:AUx}
\end{equation} 
where
\begin{align}\label{eq:AU}
  n^{(\pm)}   = \#\{ \delta_{k}^{(\pm)} \}.
\end{align} 
Of course, the arcs $\delta_{k}^{(\pm)}$ and the numbers $ n^{(\pm)}$  depend on the chosen interval $\Lambda$.
It is easy to show  (see \cite[Lemma 6.4]{SobYaf}) that 
under assumption  \eqref{eq:TH} 
the numbers $n^{(\pm)}  $ are finite.

Let us now calculate the spectral multiplicity of the multiplication 
operator $\boldsymbol\Omega$ on the interval $\Lambda$.

 \begin{theorem}\label{omega}
Suppose that $\omega$ satisfies Condition \ref{SM:cond}, and that 
an interval $\Lambda$ satisfies condition \eqref{eq:TH}. 
Let the numbers $n^{(\pm)}  $ be defined by formula \e{eq:AU}.
Then the spectral representation of  
the operator $\boldsymbol\Omega  E_{\boldsymbol\Omega} (\Lambda)$   restricted to the 
subspace $\Ran \1_{\Delta^{(\pm)}}$
is realized on the space 
$L^{2}(\Lambda; {\mathbb C}^{n^{(\pm)}})$. 
In other words, the spectral multiplicity of 
the operator $\boldsymbol\Omega\1_{\Delta^{(\pm)}}$   
on $\Lambda$ equals $ n^{(\pm)}  $.
\end{theorem}
 
 \begin{proof}
 It follows from equality \e{eq:AUx} that
  \begin{equation*}
  E_{\boldsymbol\Omega}(\Lambda) \1_{\Delta^{(\pm)}}= 
  \1 _{\omega^{-1}(\Lambda)}  \1_{\Delta^{(\pm)}}=
  \sum_{k=1}^{n^{(\pm)}} \1_{\delta_{k}^{(\pm)} }.
\end{equation*} 
According to Lemma~\ref{Omegap:lem} the spectrum of the operator $\boldsymbol\Omega   \1_{\Delta^{(\pm)}}$ on each of $ n^{(\pm)}$ disjoint arcs $\delta_{k}^{(\pm)} $ is simple. 
 Since $\omega (\delta^{(\pm)}_k ) = \Lambda$ 
for each $k = 1, 2, \dots, n^{(\pm)}$, 
the spectral multiplicity of $\boldsymbol\Omega\1_{\Delta^{(\pm)}}$ on $\Lambda$ 
equals $n^{(\pm)}$, as required. 
\end{proof}


\begin{figure}
\resizebox{12cm}{!}{
\begin{tikzpicture}
\draw [very thick] plot[smooth, tension=.7] coordinates 
{(-5,-2) (-4,-2) (-3.5,-1.5) (-3,-0.5) (-1.5,1) (-1,1.5) 
(0,2) (1,1) (2,0) (2.5,-0.5) (3,-1.5) (3.5,-2.5) (4,-3) (5.5,-3)};
\draw[very thick] plot[smooth, tension=.7]coordinates{(5.5,2.5)(6.0,3.0) (6.5,3)};
\draw[very thick] plot[smooth, tension=.7]coordinates{(6.5,1.5)(7,2)(7.5,2)};
\draw[very thick]plot[smooth, tension=.7]coordinates{(7.5,-2) (8.5,-2.5) (9,-2)};
\draw (-5,1) node (v4) {} -- (9,1);
\draw (-5,-1.5) node (v3) {} -- (9,-1.5);
\draw[ultra thin, -latex] (-5,-4) node (v1) {} -- (10,-4) node (v2) {};
\draw [ultra thin, -latex] (-5,-4) -- (-5,3.5);
\draw [double][ultra thick](-5,-1.5) --  (-5,1)
node[pos=.5, left] {$\Lambda$};
\draw[dashed] (5.5,2.5) -- (5.5,-3) node (v11) {};
\draw[dashed] (7.5,2) -- (7.5,-1.5) node (v12) {};
\draw [dashed](-3.5,-1.5) node (v16) {} -- (-3.5,-4) node (v5) {};
\draw[dashed] (-1.5,1) -- (-1.5,-4) node (v6) {};
\draw[dashed] (1,1) -- (1,-4) node (v7) {};
\draw[dashed] (3,-1.5) -- (3,-4) node (v8) {};
\draw [double][ultra thick, blue](-3.5,-4) -- (-1.5,-4)
node[pos=.5, above, black] {$\delta^{(-)}_1 $};
\draw  [double][ultra thick, blue ](1,-4) -- (3,-4)
node[pos=.5, above, black] {$\delta^{(+)}_1 $};
\draw[dashed] (v11) -- (5.5,-4) node (v15) {};
\draw[dashed] (v12) -- (7.5,-4);
\coordinate[label=below:$\eta_1$] (A) at (5.5, -4);
\coordinate[label=below:$\eta_2$] (A) at (6.5, -4);
\coordinate[label=below:$\eta_3$] (A) at (7.5, -4);
\draw [dashed](6.5,3) -- (6.5,-4);
\end{tikzpicture}}
\caption{
Example: $\omega^{-1} (\Lambda) = \delta^{(+)}_1 \cup \delta^{(-)}_1 $, 
${\sf S}^{(-)} = {\sf S}^{(-)} (\Lambda) = \{\eta_1\}$, 
${\sf S}^{(+)} = \{\eta_2, \eta_3\}$, 
${\sf S}^{(+)}  (\Lambda)= \{\eta_3\}$. }
\end{figure}
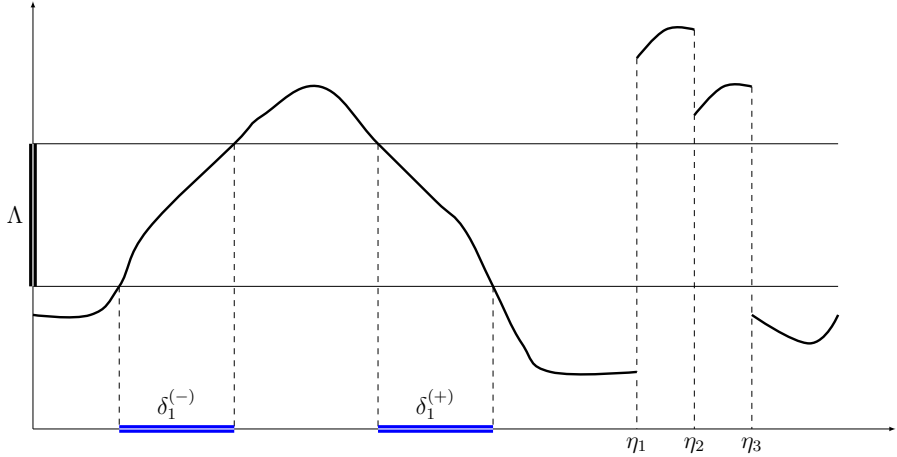

\subsection{Toeplitz operators: diagonalization} 
\label{diag:subsect} 
As explained in \cite{SobYaf}, the Toeplitz operator $T(\omega)$ 
can be diagonalized by a unitary operator $\Phi = \Phi(\omega)$ whose 
integral kernel is given by the generalized eigenfunctions of $T(\omega)$. 
In the simple spectrum case, the operator $T(\omega)$ has one generalized eigenfunction 
$\varphi(z; \lambda)$, $z\in \mathbb D$, 
a.e. $\lambda\in (\gamma_1, \gamma_2)$, and the unitary operator 
$\Phi:{\mathbb H}^2\to L^2 (\gamma_{1}, \gamma_{2})$ is defined by 
the formula 
 \begin{equation}
(\Phi f)(\lambda) = \lim_{r\to 1-0}\int_{\mathbb T}\ov{\varphi(r\z;\lambda) } f(\z) d{\bf m}(\z),
\label{eq:Ros8}
\end{equation}
which is a more precise version of \eqref{eq:Ros3}.  
The adjoint operator is given by the formula
\begin{align*}
(\Phi^* g)(z) = \int\limits_{\gamma_1}^{\gamma_2} \varphi(z; \lambda) g(\lambda) d\lambda,\q z\in\mathbb D,
\end{align*}
for all $g\in L^2(\gamma_1, \gamma_2)$. 
Thus, in view of \eqref{eq:Ros4}, the representation
 \begin{equation}
(e^{-iT t} f)(z)= \int_{\gamma_{1}}^{\gamma_{2}} 
\varphi(z;\lambda)  e^{-i\lambda t} \tilde{f}(\lambda) d\lambda,
\q \tilde{f}=\Phi f,
\label{eq:Ros7}
\end{equation}
holds.   
Here we do not discuss the eigenfunctions in greater detail since for our purposes we need the 
representation \eqref{eq:Ros7} only for the singular symbol \eqref{eq:singg}, 
for which $\varphi$ is found explicitly, see \eqref{eq:RS}.

 \subsection{Toeplitz operators: 
 piecewise continuous symbols} 

Now we turn our attention to Toeplitz operators with piecewise continuous symbols.  
Let ${\sf S}^{(\pm)}$ consist of  $\eta_k \in {\sf S}$ such that 
\begin{equation*}
\pm (\omega(\eta_k - 0)-\omega (\eta_k + 0))>0,\q \eta_k\in {\sf S}^{(\pm)},
 \end{equation*}
 and let ${\sf S}_{0}\subset {\sf S}$ be the set of those $\eta_k$ where 
$\omega(\eta_k - 0) = \omega(\eta_k + 0)$ but the derivative $\omega'(\z)$ is not continuous 
at the point $\eta_k$. 
Thus ${\sf S}$ is the disjoint union 
\begin{align}\label{SS:eq}
\sf {\sf S} = {\sf S}^{(+)}\cup {\sf S}^{(-)}\cup {\sf S}_{0}.
\end{align}
We associate with every discontinuity $\eta_k\in {\sf S}^{(\pm)}$ the interval (the ``jump"):
\begin{align}\label{Lk:eq}
\Lambda_{k}:= \Lambda_{\eta_k} 
= [\omega(\eta_k\pm 0), \omega (\eta_k \mp 0)],\q \eta_k\in {\sf S}^{(\pm)},
\end{align}
cf. \eqref{eq:Jump2}.
 The condition 
\eqref{eq:TH} guarantees that 
$\omega(\eta_{k}\pm 0)\not\in \Lambda$ for all $\eta_{k}\in {\sf S}$, and hence    
each interval 
$\Lambda_k$  either contains $\Lambda$, or is 
disjoint from $\Lambda$. 
We are interested only in those singular points $\eta_k$, for which 
$\Lambda\subset \Lambda_k$, and we denote 
this subset by ${\sf S}(\Lambda)$. Introduce also the notation 
${\sf S}^{(\pm)}(\Lambda) = {\sf S}(\Lambda)\cap {\sf S}^{(\pm)}$ 
and put 
\begin{align}
  s^{(\pm)  } = s^{(\pm)}(\Lambda) = 
\#\{ {\sf S}^{(\pm)}(\Lambda)\}.
\label{eq:countZ}\end{align} 
  Note that ${\sf S}(\Lambda) = {\sf S}^{(+)}(\Lambda)\cup{\sf S}^{(-)}(\Lambda)$.  
 All these  objects are illustrated in Fig.~2. 
 
Now we are in a position to  quote \cite[Theorems~6.5, 6.6]{SobYaf}. 

\begin{theorem}\label{Ros_smooth:thm}
Suppose that $\omega$ satisfies Condition \ref{SM:cond}, and that 
an interval $\Lambda$ satisfies condition \eqref{eq:TH}. 
Let the numbers $n^{(\pm)}= n^{(\pm)}(\Lambda) $ and $ s^{(\pm)}= s^{(\pm)}(\Lambda)$ be defined by \e{eq:AU} and \e{eq:countZ}, respectively.
Then $ n^{(+)} + s^{(+)} =  n^{(-)} + s^{(-)}$ and  the 
spectral representation of  
the operator $T(\omega)$ restricted to the 
subspace $E_{T(\omega)} (\Lambda){\mathbb H}^{2}$ 
is realized on the space 
$L^{2}(\Lambda; {\mathbb C}^m)$ where  
 \begin{align}\label{count:eq}
m = n^{(\pm)} + s^{(\pm)}
\end{align}  
$($for both signs $``\pm")$. 
In other words, the spectral multiplicity $m_{T(\omega)}(\Lambda)$ 
of $T(\omega)$ on $\Lambda$ is finite and it coincides with the number \e{count:eq}.
\end{theorem}

By this theorem, the spectrum of the 
Toeplitz operator with the symbol in Fig. 2 has multiplicity two on $\Lambda$.

\subsection{Examples} 
\label{examples:subsect}

Let us give two examples illustrating 
Theorem \ref{Ros_smooth:thm}. Both examples were mentioned in \cite{Ros2} and discussed in more detail in \cite{SobYaf}. 

\begin{example}
First we consider the regular symbol 
 \begin{equation*}
 \omega_{\rm r} (\z)= (\z+\z^{-1})/2,
 \end{equation*}
 for which $\sigma(T(\omega_{\rm r}))= [-1,1]$. 
Now we have ${\sf S}= \varnothing$, 
$\Delta_0 = 
\{-1, 1\}\subset {\mathbb T}$,
 $\Lambda_{\rm exc}= \Lambda_{\rm cr}= \{-1, 1\}\subset {\mathbb R}$, $\Delta^{(+)}=(1,-1) \subset \mathbb T$, 
 $\Delta^{(-)}=(-1,1)  \subset \mathbb T$, $n^{(\pm)}=1$, $s^{(\pm)}=0$. By Theorem \ref{Ros_smooth:thm}, the spectrum of the operator $T(\omega_{\rm r})$ is simple, and, 
 according to Definition \ref{class}, it is thick. 
 
 The eigenfunctions of the Toeplitz operator $T(\omega_{\rm r})$ equal 
\begin{equation}
\varphi_{\rm r}(z;\lambda) =\sqrt{\frac{2}{\pi}}\frac{(1-\lambda^{2})^{\frac{1}{4}}}{1-2\lambda z + z^{2}}, \q \lambda\in (-1,1).
\label{eq:Reg}
\end{equation}
They  are singular at the points 
\[z_\pm = \lambda\pm i\sqrt{1-\lambda^{2}}.
\]
Note that function \e{eq:Reg} belongs to the Hardy class ${\mathbb H}^p$ for any $p<1$ but $\varphi_{\rm r}(\cdot;\lambda)\not\in{\mathbb H}^1$.
\end{example}
 
 \begin{example}
 The simplest singular symbol is given by  the indicator 
 \begin{equation}\label{eq:singg}
\omega_{\rm s} (\z)=\1_{(\z_{1}, \z_{2})} (\z)
 \end{equation}
of an arc $(\z_{1}, \z_{2})\subset {\mathbb T}$; 
then $\sigma(T(\omega_{\rm s}))= [0,1]$.
Now we have ${\sf S}= \{\z_{1}, \z_{2} \}$, 
${\sf S}^{(+)}= \{  \z_{2} \}$, ${\sf S}^{(-)}= \{  \z_1 \}$, $s^{(\pm)}=1$
and $\Lambda_{\rm exc}= \Lambda_{\rm thr}= \{0,1 \}$, 
 $\Delta^{(\pm)}=\varnothing$,  $n^{(\pm)}=0$. 
 By Theorem \ref{Ros_smooth:thm}, the spectrum of 
 the operator $T(\omega_{\rm s})$ is again simple, and, according to Definition \ref{class}, it is thin. 
 
 The eigenfunctions of the Toeplitz operator $T(\omega_{\rm s})$ equal
 \begin{equation}
\varphi_{\rm s} (z;\lambda) = \varkappa(\lambda)  (1-z / \z_1)^{-1/2-i\sigma (\lambda)} 
(1-z/\z_{2})^{-1/2+i\sigma (\lambda)},
\label{eq:RS}
\end{equation}
where 
 \begin{equation}
 \sigma (\lambda)= \frac{1}{2\pi}\ln(\lambda^{-1}-1)
\label{eq:RS1}
\end{equation}
and 
 \begin{equation}
 \varkappa (\lambda)= \sqrt{\frac{|\z_{2}-\z_{1}|}{2\pi\lambda}}
e^{-\pi\sigma (\lambda){\mathbf m}(\z_1,  \z_2)}.
%
\label{eq:RS2}\end{equation} 
 \end{example}

As mentioned in the Introduction,  formulas \eqref{eq:Reg} and \eqref{eq:RS} show that the  
eigenfunctions  of the Toeplitz operators have stronger singularities in the smooth case, and 
that   in the singular case the location of these singularities 
is independent of $\lambda$.


\section{Toeplitz versus multiplication operators. Smooth symbols}
\label{mult:sect}

\subsection{Scattering theory. Basic notions}\label{ST:subsect}
We refer to the books \cite{RS4}  or \cite{Yafaev}  for  detailed information.

Let $A$ and $B$ be self-adjoint operators in, possibly, 
different Hilbert spaces $\mathcal{ H}$ and $\mathcal{ G}$, respectively, and 
 let $J: \mathcal{H}\to\mathcal{G}$ be a bounded operator.

\begin{definition} 
Under the assumption of their existence,
the strong limits
\begin{equation}
\slim_{t\to\pm\infty} e^{i Bt} Je^{- i At} P^{\rm(ac)}_{A} =: W_{\pm} (B,A;J)    
 \label{eq:WO}\end{equation}
are called the wave operators  for the pair of 
 self-adjoint operators $A$ and $B$ and the  ``identification" $J$. 
 \end{definition} 
 
Here and below all statements about  the wave operators $W_{\pm}(B, A; J)$ 
(as well as about other objects labelled by $``{\pm} "$) are understood as two independent statements.   
 
 Obviously, the wave operator $W_{\pm} (B,A;J)$ exists if the limit $e^{i Bt} Je^{- i At} f$   as $t\to\pm\infty$ exists  on a set of vectors $f$ dense in the subspace $\mathcal{ H}^{\rm(ac)}_{A} $.
 
 Let us list some elementary properties of wave operators following from the mere fact of their existence:
 \begin{enumerate}[\rm(a)]
\item
The wave operators are bounded and
\[
\| W_{\pm} (B,A;J) \|\leq \| J\|.
\]

\item
The wave operators enjoy the intertwining property
\begin{align}\label{intertwine:eq}
E_{B}(X)W_{\pm} (B,A;J) = W_{\pm} (B,A;J) E_{A}(X),
\end{align}
where $X \subset {\mathbb R}$ is an arbitrary Borel set.

\item
The ranges $\Ran W_{\pm} (B,A;J)$ of the wave operators satisfy
 the inclusions 
\begin{equation}
\Ran W_{\pm} (B,A;J)\subset \mathcal{ G}^{\rm(ac)}_{B} ,
\label{eq:WO1} \end{equation}
 and their closures
      are invariant subspaces of the operator $B$.

\item
If
\begin{equation}
\lim_{t\to\pm \infty}\| J e^{- i At} P^{\rm(ac)}_{A} f\| =\| P^{\rm(ac)}_{A} f\|,
\label{eq:WOac}\end{equation}
then $\|W_{\pm} (B,A;J) f\| = \| P^{\rm(ac)}_{A} f\|$. Thus,
the wave operator $W_{\pm} (B,A;J)$ is isometric on $\mathcal{H}^{\rm(ac)}_{A} $ if
condition \e{eq:WOac} is satisfied for 
a set of vectors $f$ dense in   $\mathcal{ H}^{\rm(ac)}_{A}$. 
In particular, for $\mathcal H = \mathcal G$ 
the wave operators $W_{\pm}(B, A; I)$ are isometric on $\mathcal{H}^{\rm(ac)}_{A} $.

\item
If the  operator $W_{\pm} (B,A;J)$ is isometric on $\mathcal{H}^{\rm(ac)}_{A} $, then
   it follows from (b) that the restriction of the operator $B$ to the subspace 
$\Ran W_{\pm}(B, A; J)$ is unitarily equivalent to the absolutely continuous part 
$A^{\rm(ac)}$ of the operator $A$.

 \item
If $f= W_{\pm} (B,A;J) f_{0}$, then 
\begin{equation}
\lim_{t\to\pm \infty}\| e^{-i Bt} f -J e^{-i At} P^{\rm(ac)}_{A} f_{0}\|=0.
 \label{eq:WO1x}\end{equation}
 For a given $f\in \Ran W_{\pm} (B,A;J) $, the choice of $f_{0}$ 
 satisfying \e{eq:WO1x} is in general not unique.
 One can set $f_{0}= W_{\pm} (B,A;J)^* f$ if the  operator $W_{\pm} (B,A;J)$ is isometric on 
 $\mathcal{H}^{\rm(ac)}_{A} $.
 
    \item 
If $J$ is compact, then $J e^{-iAt}f\to 0$ as $ t\to \pm\infty$ for all 
$f\in\mathcal{H}^{\rm(ac)}_{A}$, so that 
$W_{\pm} (B,A;J) = 0$.
 
    \end{enumerate}

\begin{definition} 
A  wave operator  $W_{\pm} (B,A;J)$ is called complete if the equality holds in \e{eq:WO1}:
\begin{equation*}
\Ran W_{\pm} (B,A;J) = \mathcal{G}^{\rm(ac)}_{B} .
 \end{equation*}
 \end{definition}
 
  If the  operator  $W_{\pm} (B,A;J)$ is isometric on 
 $\mathcal{H}^{\rm(ac)}_{A} $ and complete, then the  absolutely continuous parts $A^{\rm(ac)}  $ and $B^{\rm(ac)}  $ of the operators $A$ and $B$
are unitarily equivalent.
 
Let us note an important special case  $\mathcal{H} = \mathcal{G}$, 
$J=I$ (the identity operator). For short, we  write $W_{\pm}(B, A; I) = W_{\pm}(B, A)$. 
Of course, the operators 
$W_{\pm} (B, A)$ are isometric on the subspace ${\mathcal H}^{\rm(ac)}_{A}$. 
The completeness of $W_{\pm} (B,A)$ is equivalent to the existence of $W_{\pm} (A,B)$, and in this case  
\begin{align*}
 W_{\pm}(A, B) = W_{\pm}^*(B, A).
\end{align*}


\subsection{Existence of wave operators}


 The following condition (known as Cook's criterion) 
 of the existence  of wave operators \e{eq:WO} is quite elementary, but it requires 
 the knowledge of the ``free" evolution $e^{-iAt}$. In view of our applications, we assume 
 that $A$ and $B$ are bounded operators. 
 
\begin{theorem} \label{Cook}
Suppose that the operator $A$ is absolutely continuous and 
\[
\int_{0}^{\pm\infty} \| (BJ-J A)e^{-iAt} f\| dt <\infty
\]
for a set of elements 
$f$ dense in $\mathcal{H}$. Then the corresponding wave operator $W_{\pm} (B, A; J)$ exists.
\end{theorem}
 
On the contrary, the trace class method treats the operators $A$ and $B$ on an equal footing. 
The ideal of trace class operators is denoted by $\mathfrak{S}_{1}$. First, we recall the classical  Kato - Rosenblum theorem.

 \begin{theorem} \label{KR} 
 Suppose that $\mathcal{ H}=\mathcal{ G}$, $J=I$ and
  $ B- A\in \mathfrak { S}_{1}$.
 Then the wave operators  $W_{\pm}(B, A)$ exist, are isometric on $\mathcal{ H}^{\rm(ac)}_{A}$ and are complete, that is, 
 \[
 \Ran W_{\pm} (B, A ) 
=\mathcal{ H}^{\rm(ac)}_B.
\]
 \end{theorem}

An extension of this result to arbitrary $J$ is due to D.~Pearson.
 
 \begin{theorem} \label{WO2} 
 Suppose that
 \begin{equation*}
   BJ- JA\in \mathfrak { S}_{1}.
 \end{equation*}
 Then the wave operators $W_{\pm} (B, A; J)$ exist.
 \end{theorem}

\subsection{Duplex Toeplitz operators} 

Let us return to Toeplitz operators.
Here we compare the Toeplitz operator $T = T(\omega)$ defined by \e{eq:TT}
with the multiplication operator $\boldsymbol\Omega$ by 
 a bounded  function $\omega$,
\begin{align*}
(\boldsymbol\Omega f)(\z) = \omega(\z) f(\z) ,\q \z\in \mathbb T, 
\end{align*}
acting on the space $\mathcal H = L^2(\mathbb T)$. In this subsection $\omega$ may be complex-valued.
Recall that ${\mathbb P}_{\pm}$ 
are the orthogonal projections in $L^2(\mathbb T)$ onto the Hardy spaces 
${\mathbb H}_{\pm}^2 (\mathbb T)$. 

It is convenient to introduce on $L^2 (\mathbb T)$ the 
\textit{duplex} Toeplitz 
operator 
\begin{equation}
{\bf T } = \mathbf T(\omega) 
= {\mathbb P}_+ \boldsymbol\Omega {\mathbb P}_+ 
+ {\mathbb P}_- \boldsymbol\Omega {\mathbb P}_-
\label{eq:T} \end{equation}
and the \textit{symmetrized}  
Hankel operator 
\begin{equation}
{\bf H } ={\bf H } (\omega) = {\mathbb P}_+ \boldsymbol\Omega {\mathbb P}_- 
+ {\mathbb P}_- \boldsymbol\Omega {\mathbb P}_+
\label{eq:H} \end{equation}
(these operators are symmetric if $\omega = \overline{\omega}$).
Then
\begin{equation}
\boldsymbol\Omega = {\bf T }+  {\bf H }.
\label{eq:HtH} \end{equation}
Under very general assumptions on its symbol $\omega$ the
  operator $\mathbf H$ is trace class, 
so that
${\bf T }$
can be viewed 
as a perturbation of the multiplication operator $\boldsymbol\Omega$.

Let $T_\pm =T_\pm  (\omega)$ be the restriction of the operator $\mathbf T$  
 onto the subspace ${\mathbb H}_\pm^2 (\mathbb T)$. Clearly, $T = T_{+}$. 
  Define the unitary operator $V$ on $ L^2 (\mathbb T) $ by the formula 
\[
(  Vf)(\z) = \bar{\z}f(\bar{\z}),\q \z\in \mathbb T.
\]
Evidently, $V: {\mathbb H}_\pm^2 (\mathbb T) \to {\mathbb H}_\mp ^2 (\mathbb T)$ and $V {\mathbb P}_{\pm}  = {\mathbb P}_{\mp} V$. 
 The following fact is almost obvious.
 
 \begin{lemma} 
Let $\tilde \omega(\z) = \omega(\overline \z)$.
 Then   
\begin{equation}
V T_-  (\omega)f= T_{+} (\tilde \omega) Vf, \q \forall f\in {\mathbb H}_-^2 (\mathbb T) ,
\label{eq:HtHx} \end{equation}
so that the operators $T_-  (\omega)$ and $ T_{+} (\tilde \omega)$ are unitarily equivalent.
\end{lemma}

\begin{proof}
Since $V (\omega g)= \tilde \omega Vg$ for all $g\in L^2 (\mathbb T) $, we see that
\[
V {\mathbb P}_{-}(\omega {\mathbb P}_{-} f)=  {\mathbb P}_{+}V (\omega {\mathbb P}_{-} f)=
{\mathbb P}_{+} (\tilde\omega V{\mathbb P}_{-} f)
= {\mathbb P}_{+} (\tilde\omega {\mathbb P}_{+} Vf), \q \forall f\in L^2 (\mathbb T) ,
\]
whence \e{eq:HtHx} follows.
\end{proof}

Thus, for  a real-valued bounded $\omega$, the 
operators  $T_+(\omega)$ and $T_-(\omega) = V^* T_+(\tilde\omega)V$   
are absolutely continuous on $\mathbb H^2_+$ 
and $\mathbb H^2_-$ respectively, 
and $\mathbf T$ is absolutely continuous on $L^2(\mathbb T)$. 

\subsection{Smooth symbols} 

Recall one of the equivalent definitions of the Besov 
class ${\mathbb B}^1_{11}={\mathbb B}^1_{11}({\mathbb T})$:  a function $\omega\in {\mathbb B}^1_{11}$ if
\begin{align*}
\int_{\mathbb T} \int_{\mathbb T}  |\eta-1|^{-2}
 |\omega (\z \eta) + \omega (\z \eta^{-1}) - 2 \omega (\z)|d{\mathbf m}(\z) 
d{\mathbf m}(\eta)<\infty.
\end{align*}
A list of basic properties of the Besov spaces can be found 
in the book \cite[Appendix 2.6]{Pe}.
For us,  it is useful to remember that 
\[
\mathbb W^{2,1}(\mathbb T)\subset \mathbb B^1_{11}
\subset \mathbb W^{1,1}(\mathbb T)\subset C(\mathbb T)
\] 
and that  
%
$\mathbb P_\pm \mathbb B^1_{11}\subset \mathbb B^1_{11}$. 

We need the   well known criterion of V.~V.~Peller 
 for a Hankel operator to be trace class. 
 The function  $\omega(\z)$ is not required  to be real-valued here.

 
\begin{theorem} \label{PE}\cite[Theorem~6.1.1]{Pe}
The inclusion
$ {\mathbb P}_\pm \boldsymbol\Omega  {\mathbb P}_\mp  \in\mathfrak{S}_{1} $
 holds
 if and only if $ {\mathbb P}_\pm \omega \in {\mathbb B}^1_{11}$.
\end{theorem} 

 \begin{corollary} 
The  symmetrized Hankel operator  \eqref{eq:H} satisfies 
\begin{equation}
   {\mathbb P}_- \boldsymbol\Omega  {\mathbb P}_+ + {\mathbb P}_ 
   + \boldsymbol\Omega {\mathbb P}_- \in\mathfrak{S}_{1}
\label{eq:Hs1}
\end{equation}
 if and only if $\omega \in {\mathbb B}^1_{11}$.
\end{corollary}

 \begin{proof}
 If  $\omega \in {\mathbb B}^1_{11}$, then ${\mathbb P}_{-} \omega \in {\mathbb B}^1_{11}$ and ${\mathbb P}_{+}  \omega  \in {\mathbb B}^1_{11}$ so that 
  \begin{equation}
{\mathbb P}_- \boldsymbol\Omega {\mathbb P}_+ \in\mathfrak{S}_{1}
\q {\rm and}\q {\mathbb P}_+ \boldsymbol\Omega  {\mathbb P}_- \in\mathfrak{S}_{1}
\label{eq:Hs2}\end{equation}
 by Theorem~\ref{PE}. Conversely, inclusion \e{eq:Hs1} implies both inclusions 
 \e{eq:Hs2}, and hence ${\mathbb P}_{-}\omega \in {\mathbb B}^1_{11}$ 
 and ${\mathbb P}_{+} \omega \in {\mathbb B}^1_{11}$ again by 
 Theorem~\ref{PE}.
  \end{proof}
  
Putting together Theorems~\ref{KR}   and \ref{PE}, we obtain the following 
result for the  self-adjoint duplex Toeplitz operator $\mathbf T(\omega)$. Recall that this operator is absolutely continuous.
  
\begin{theorem} \label{TrCl} 
Suppose that $\omega =\bar{\omega}\in {\mathbb B}^1_{11}$. 
Then the wave operators $W_{\pm} ({\bf T }(\omega), \boldsymbol\Omega)$ exist, are isometric on the absolutely continuous subspace of $\boldsymbol\Omega$ and are complete, that is, 
 \begin{equation*}
\Ran W_\pm (\mathbf T(\omega), \boldsymbol\Omega) = L^2(\mathbb T).
\end{equation*}
\end{theorem}
 
In view of the general property \eqref{intertwine:eq}, we also have

  \begin{corollary} 
  The intertwining property
 \begin{equation*}
{\bf T }(\omega)W_{\pm} ({\bf T }(\omega), \boldsymbol\Omega)= W_{\pm} ({\bf T }(\omega), \boldsymbol\Omega)\,\boldsymbol\Omega
\end{equation*}
holds, and the operator ${\bf T }(\omega)$ is unitarily 
  equivalent to the absolutely continuous part ${\boldsymbol\Omega}^{({\rm ac})}$
   of the operator $\boldsymbol\Omega$.
\end{corollary}
 
 Theorem~\ref{TrCl} leads to the following result for  
the Toeplitz operators $T_{\pm} (\omega)={\mathbb P}_{\pm}{\bf T }(\omega) {\mathbb P}_{\pm}$ on the subspaces ${\mathbb H}^2_{\pm} ({\mathbb T})$.
 
  \begin{corollary} \label{TrCl1} Let 
 $\omega=\bar{\omega} \in {\mathbb B}^1_{11}$.
 Then, for both signs $``\pm"$, the  wave operators
 \begin{equation}
  W_{\pm} (T_+ (\omega), \boldsymbol\Omega;  {\mathbb P}_+)\q {\rm and }\q
 W_{\pm} (T_-(\omega), \boldsymbol\Omega; {\mathbb P}_-)
\label{eq:wo}\end{equation}
  exist
 and
  \begin{equation}
 W_{\pm} ({\bf T }(\omega), \boldsymbol\Omega) 
 = W_{\pm} (T_{+}(\omega), \boldsymbol\Omega; {\mathbb P}_{+})+
 W_{\pm} (T_{-}(\omega), \boldsymbol\Omega; {\mathbb P}_{-}).
\label{eq:WT}
\end{equation}
Furthermore, the relations
\begin{align}\label{WOcompl:eq}
\Ran W_{\pm} (T_{+}(\omega), \boldsymbol\Omega; {\mathbb P}_{+}) = \mathbb H^2_+
\q {\rm and }\q
\Ran W_{\pm} (T_{-}(\omega), \boldsymbol\Omega; {\mathbb P}_{-}) = \mathbb H^2_- 
\end{align}
hold true.
 \end{corollary}

\begin{proof} 
In view of definition \eqref{eq:T}, 
the existence of wave operators on the right-hand side of \eqref{eq:WT} and the formula 
\eqref{eq:WT} itself follow from Theorem \ref{TrCl} due to the identities $e^{i{\bf T}t}{\mathbb P}_{\pm}= e^{i T_{\pm}t}{\mathbb P}_{\pm}$ and
\[
\mathbb P_+ W_\pm (\mathbf T, \boldsymbol\Omega) = W_{\pm} (T_{+}(\omega), \boldsymbol\Omega; {\mathbb P}_{+}), \q
\mathbb P_- W_\pm (\mathbf T, \boldsymbol\Omega) = W_{\pm} (T_{-}(\omega), \boldsymbol\Omega; {\mathbb P}_{-}).
\]
These relations, together with the completeness of $W_\pm (\mathbf T, \boldsymbol\Omega)$,  imply \e{WOcompl:eq}.
   \end{proof}

  According to  \eqref{intertwine:eq}, 
all wave operators constructed above possess
  the intertwining property, for example,
\begin{equation*}
 T_{+}(\omega)W_{\pm} (T_{+}(\omega), \boldsymbol\Omega ; \mathbb P_+) 
 = W_{\pm} (T_{+}(\omega), \boldsymbol\Omega; \mathbb P_+)\,\boldsymbol\Omega.
\end{equation*}
The isometry of the wave operators \eqref{eq:wo} 
is discussed in Section~3.6 in a  more general setting.
     
\subsection{Local smoothness.} 

In the following  assertion we do not assume that the symbol $\omega$ is smooth on the whole unit circle ${\mathbb T}$.
We first consider the duplex Toeplitz operator ${\bf T }$ defined by \e{eq:T}.


\begin{theorem} \label{Trloc} 
 Let $\omega=\bar{\omega} \in L^\infty ({\mathbb T}) $ and $\varphi\omega\in {\mathbb B}^1_{11}(\mathbb T)$ for some function  $\varphi\in C ^\infty ({\mathbb T})$. Denote by $J_{\varphi}$ 
  the operator of multiplication by $\varphi$ in the space $L^2(\mathbb T)$. 
 Then the wave operators $W_{\pm} ({\bf T }(\omega), \boldsymbol\Omega; J_{\varphi})$ exist.
\end{theorem}
 
  \begin{proof} 
   According to Theorem~\ref{WO2} it suffices to check  that
    \begin{equation}
 {\bf T }  J_{\varphi} -J_{\varphi} \boldsymbol\Omega \in\mathfrak{S}_{1}, \q {\bf T } ={\bf T }(\omega).
\label{eq:HT}\end{equation}
By \e{eq:HtH}, we have
\[
 {\bf T }  J_{\varphi} -J_{\varphi} \boldsymbol\Omega= (\boldsymbol\Omega - {\bf H } )  J_{\varphi} -J_{\varphi} \boldsymbol\Omega= - {\bf H }   J_{\varphi}.
\]
Using definition \e{eq:H} of the operator $  {\bf H} $, we now see that
\e{eq:HT} 
is equivalent to the inclusion
   \begin{equation}
   {\bf H} J_{\varphi}  =
  {\mathbb P}_{+} \boldsymbol\Omega  {\mathbb P}_{-}J_{\varphi}  
  +   {\mathbb P}_{-}\boldsymbol\Omega  {\mathbb P}_{+}J_{\varphi} \in\mathfrak{S}_{1}.
\label{eq:HT1}\end{equation}
Two terms here are quite similar.
Consider, for example, the first  one:
 \begin{equation}
  {\mathbb P}_{+} \boldsymbol\Omega  {\mathbb P}_{-}J_{\varphi}  
  =   {\mathbb P}_{+} \boldsymbol\Omega J_{\varphi}  {\mathbb P}_{-}  + {\mathbb P}_{+} \boldsymbol\Omega ( {\mathbb P}_{-}J_{\varphi}   -J_{\varphi} {\mathbb P}_{-}).
\label{eq:HTk}\end{equation}
According to Theorem \ref{PE} the operator ${\mathbb P}_{+}( \boldsymbol\Omega J_{\varphi})  {\mathbb P}_{-}  \in\mathfrak{S}_{1}$ because $\omega \varphi\in {\mathbb B}^1_{11}$. Since $  \varphi\in C^\infty$, Theorem \ref{PE} also implies that
\begin{equation}
 {\mathbb P}_{-}J_{\varphi} -  J_{\varphi} {\mathbb P}_{-} 
 =  {\mathbb P}_{-}J_{\varphi}  {\mathbb P}_{+}-   {\mathbb P}_{+}J_{\varphi}  {\mathbb P}_{-}\in\mathfrak{S}_{1}.
\label{eq:comm}
\end{equation}
 It follows that the operator \e{eq:HTk} is trace class which yields inclusions
  \e{eq:HT1} and hence \e{eq:HT}.
 \end{proof}
 
 Similarly to Corollary~\ref{TrCl1}, we  have
    
\begin{corollary} \label{TrCl2} 
Under the assumptions of Theorem~\ref{Trloc}
 the  wave operators 
\begin{equation*}
W_{\pm} (T_+ (\omega), \boldsymbol\Omega; {\mathbb P}_+  J_{\varphi})\q \mbox{and} \q
W_{\pm} (T_- (\omega), \boldsymbol\Omega; {\mathbb P}_-  J_{\varphi})
 \end{equation*}
 exist  for both signs $``\pm"$ and
\begin{equation}
W_{\pm} ({\bf T }(\omega), \boldsymbol\Omega; J_{\varphi}) 
 = W_{\pm} (T_{+}(\omega), \boldsymbol\Omega; {\mathbb P}_{+} J_{\varphi})+
W_{\pm} (T_{-}(\omega), \boldsymbol\Omega; {\mathbb P}_{-}  J_{\varphi}).
 \label{eq:HT2}\end{equation}
 \end{corollary}
  
  Let us, finally, replace $\varphi\in C ^\infty ({\mathbb T})$ by the characteristic function of an   arc $\Delta$.
 
\begin{theorem} \label{YX1} 
 Let $\omega=\bar{\omega} \in L^\infty ({\mathbb T}) $, and let $\Delta\subset\mathbb T$ be an open arc. Suppose that $\varphi\omega\in {\mathbb B}^1_{11}(\mathbb T)$ for all functions  $\varphi\in C_{0} ^\infty (\Delta)$.
Then the ``local" wave operators $W_\pm ( T(\omega), \boldsymbol\Omega; \mathbb P\1_{\Delta})$ exist. 
\end{theorem}

  \begin{proof} 
 The wave operators  
 $W_\pm ( T(\omega), \boldsymbol\Omega; \mathbb P J_{\varphi})$ exist due to 
Corollary~\ref{TrCl2}. 
The operator $\boldsymbol\Omega$ commutes with $J_\varphi$ and $\1_{\Delta}$, and hence so does  
$P^{(\rm ac)}_{\boldsymbol\Omega}$. 
As a consequence, 
\begin{align*}
W_\pm\big(T(\omega), \boldsymbol\Omega; \mathbb P J_\varphi\big) = 
W_\pm\big(T(\omega), \boldsymbol\Omega; \mathbb P\1_{\Delta}\big)J_\varphi.
\end{align*}
Since $C^\infty_0(\Delta)$ is dense in $L^2(\Delta)$, the existence of 
$W_\pm\big(T(\omega), \boldsymbol\Omega; \mathbb P\1_{\Delta}\big)$ follows.
  \end{proof}

 \subsection{Isometry of wave operators}

We start with a  relatively standard (cf., for example, \cite[Lemma~2.6.4]{Yafaev})
analytic result. Recall that $ \mathbb W^{2, 1}_{\rm loc}(\Delta)\subset C^1 (\Delta)$. As usual, we use the notation $w(x)=\omega(e^{ix}) $, $w'(x)=\omega'(e^{ix}) $.

\begin{lemma} 
Let $\Delta\subset \mathbb T$ be an arc.      
Suppose that $\omega \in \mathbb W^{2, 1}_{\rm loc}(\Delta)$  
and $\omega'(\z)\not= 0$ for $\z\in \Delta$. Then 
   \begin{equation}
\slim_{t\to\pm\infty}{\mathbb P}_+ e^{-i\boldsymbol\Omega t} \1_{\Delta}=0 \q\mbox{if}\q\pm  
\omega'(\z) >0  \q\mbox{for}\q \z\in \Delta, 
 \label{eq:HT3}\end{equation}
 and
   \begin{equation}
\slim_{t\to\pm\infty}{\mathbb P}_- e^{-i\boldsymbol\Omega t} \1_{\Delta}=0 \q\mbox{if}\q \pm
\omega'(\z) <0 \q\mbox{for}\q \z\in \Delta. 
 \label{eq:HT3-}\end{equation}
 \end{lemma}
 
\begin{proof}
Let us check, for example, \e{eq:HT3}. 
Let $F: L^2 ({\mathbb T})\to \ell^2 ({\mathbb Z})$ 
be the discrete Fourier transform defined in  
\eqref{fourier:eq}. For all $u\in L^2 ({\mathbb T})$, we have
\begin{equation}
\| {\mathbb P}_+  u\|^2=\sum_{n=0}^\infty |(F u)_{n}|^2 .
\label{eq:HT4}
\end{equation}
Let us use the representation
\begin{equation*}
(F e^{-i\boldsymbol\Omega t} \1_{\Delta}f)_{n} 
 = \int_{\Delta} \z^{-n} e^{-i\omega (\z) t} f(\z) d{\bf m} (\z) 
 = (2\pi)^{-1}\int_{X} e^{-in x - i w(x)t} {\tt f} (x)dx,
 \end{equation*}
 where $f\in L^2(\mathbb T)$,  ${\tt f} (x)= f(e^{ix})$ and $X\subset {\mathbb R}/(2\pi  {\mathbb Z})$ is an interval such that $e^{iX} = \Delta$.
Assuming that $f\in C^\infty_0(\Delta)$, that is, $ {\tt f} \in C^\infty_0(X)$,  we integrate by parts to get
\begin{equation}
 (F e^{-i\boldsymbol\Omega t} f)_{n} =-i (2\pi)^{-1} \int_{X} 
 e^{-in x - i w (x)t} \Big(\frac{{\tt f}(x)}{n+w ' (x)t} \Big)'dx.
\label{eq:HTa} \end{equation}
 Using that    
 $\pm w'(x)\geq c>0$ on ${\rm supp}\,{\tt f}$ and $\pm t >0$, we obtain the bound
 \[
 n+w' (x)t \geq  n+c|t|.
 \]
Furthermore, the condition $\omega\in \mathbb W^{2,1}_{\rm loc}(\Delta)$ ensures that  $w\in \mathbb W^{2,1}_{\rm loc}(X)$ and hence $w'' \in L^1_{\rm loc}(X)$. Thus  \e{eq:HTa} yields an estimate
   \begin{equation*}
|(F e^{-i\boldsymbol\Omega t} f)_{n}| \leq C (n+ |t|)^{-1}.
 \end{equation*}
 Substituting this estimate into \e{eq:HT4}  where 
 $u=e^{-i\boldsymbol\Omega t} f$, we see that
   \begin{equation*}
 \|{\mathbb P}_+  e^{-i\boldsymbol\Omega t} f \|^2\leq C\sum_{n=0}^\infty  (n+ |t| )^{-2}\leq C_{1} |t|^{-1}.
 \end{equation*} 
 Since $C^\infty_0(\Delta)$ is dense in $L^2(\Delta)$, this leads to \eqref{eq:HT3}.  
Relation \e{eq:HT3-} can be checked quite similarly.
       \end{proof}
        
        Let us discuss now the isometry of the   ``local" wave operators 
 $W_\pm(T, \boldsymbol\Omega; \mathbb P\1_{\Delta})$ where, as usual, ${\mathbb P} = {\mathbb P}_{+}$. Under the assumptions below their existence follows from Theorem~\ref{YX1} because  
 $\varphi \omega    \in {\mathbb B}_{1,1}^1 (\mathbb T)$ if 
 $\omega\in {\mathbb W}^{2, 1}_{\rm loc}(\Delta)$ and $\varphi\in  C_{0}^\infty (\Delta)$.


\begin{theorem} \label{XY1} 
Let $\omega =\bar{\omega} \in \mathbb W^{2, 1}_{\rm loc}(\Delta)$ 
for some arc $\Delta\subset \mathbb T$. 
If $\pm\omega'(\z) < 0$ for $\z\in\Delta$, then 
the operator $W_\pm ( T(\omega), \boldsymbol\Omega; \mathbb P\1_{\Delta})$ 
is isometric on $\Ran \1_{\Delta}$ and
\begin{equation}
 W_{\mp} (T  (\omega), \boldsymbol\Omega; {\mathbb P} \1_{\Delta}) = 0.
 \label{eq:HT9}
\end{equation}  
\end{theorem}
 
 \begin{proof}
Equality \e{eq:HT9}  follows directly from 
\e{eq:HT3}. Similarly,   relations  \e{eq:HT3-} imply that
$
  W_{\pm} (T_{-}  (\omega), \boldsymbol\Omega; {\mathbb P}_- \1_{\Delta})=0.
$
Since (cf. \e{eq:HT2})
\[
W_{\pm} ({\bf T }(\omega), \boldsymbol\Omega; \1_{\Delta}) 
 = W_{\pm} (T_{+}(\omega), \boldsymbol\Omega; {\mathbb P}_{+} \1_{\Delta})+
W_{\pm} (T_{-}(\omega), \boldsymbol\Omega; {\mathbb P}_{-}  \1_{\Delta}), 
 \]
we see that 
\begin{align*}
  W_\pm ( T(\omega), \boldsymbol\Omega; \mathbb P\1_{\Delta}) =W_\pm (\mathbf T (\omega), \boldsymbol\Omega; \1_{\Delta}) .
\end{align*}
By 
property (d) of wave operators (see Section~3.1),
the operator on the right
is isometric on ${\rm Ran}\, \1_\Delta$   because   $\1_\Delta   P_{\boldsymbol\Omega}^{\rm(ac)}=\1_\Delta  $ according to Lemma~\ref{Omegap:lem}.
Hence the same is true for the operator on the left.
  \end{proof}
    
\subsection{Main result for smooth symbols.} 

The results obtained so far can be summarized in the following two theorems.
Recall that the sets $\Delta^{(\pm)}$ were defined in \e{eq:dpm}.
  
\begin{theorem}\label{deltapm:thm} 
 Suppose that $\omega=\bar{\omega}\in L^\infty ({\mathbb T})$ and that
$\omega\in \mathbb W^{2, 1}_{\rm loc}(\Delta)$  for some open set $\Delta\subset \mathbb T$ of full measure.  
Then
\begin{enumerate}[\rm(i)]

\item
the wave operators $W_\pm (T (\omega), \boldsymbol\Omega; \mathbb P)$ exist, and the intertwining property 
\begin{equation*}
T (\omega)W_{\pm} (T(\omega), \boldsymbol\Omega; \mathbb P) 
=  W_{\pm} (T(\omega), \boldsymbol\Omega; \mathbb P)\,\boldsymbol\Omega 
\end{equation*}
 holds.

\item
The operators $W_\pm (T (\omega), \boldsymbol\Omega; \mathbb P)$ 
 are isometric on $\Ran \1_{\Delta^{(\pm)}}$ or, equivalently, 
\begin{equation}
 W_{\pm}^* (T(\omega), \boldsymbol\Omega; \mathbb P) W_{\pm} (T(\omega), \boldsymbol\Omega; \mathbb P)= \1_{\Delta^{(\pm)}}.
\label{eq:isom}\end{equation}
    Moreover, the relations 
\begin{align}\label{deltapm:eq}
W_\pm (T (\omega), \boldsymbol\Omega; \mathbb P\1_{\Delta^{(\pm)}})=W_\pm (T (\omega), \boldsymbol\Omega; \mathbb P)  \q \mbox{and} \q  W_\pm (T (\omega), \boldsymbol\Omega; \mathbb P\1_{\Delta^{(\mp)}})=0
\end{align} 
are satisfied. 

\item
The restriction of the Toeplitz operator $T (\omega)$ to the subspace
\begin{equation*}
 \mathcal H^{(\pm)}  
 = \Ran W_\pm (T (\omega), \boldsymbol\Omega; {\mathbb P} )
\end{equation*}
of ${\mathbb H}^2 $
is unitarily equivalent to the  multiplication operator $\boldsymbol\Omega$ on 
the  subspace $\Ran \1_{\Delta^{(\pm)}}$  of $L^2 ({\mathbb T} )$. 
\end{enumerate}
\end{theorem}

\begin{proof} 
The wave operators $W_\pm (T (\omega), \boldsymbol\Omega; \mathbb P \1_{\Delta} )$ exist by 
Theorem~\ref{XY1}.
This implies the existence of $W_\pm (T (\omega), \boldsymbol\Omega; \mathbb P )$  
because the set $\Delta$ has full measure. The intertwining property is a direct consequence of the existence of wave operators, see \eqref{intertwine:eq}.
By \e{eq:ac1} we have the equality
\begin{align*}
W_\pm (T, \boldsymbol\Omega; \mathbb P) = W_\pm (T, \boldsymbol\Omega; \mathbb P\1_{\Delta^{(+)}})
+ W_\pm (T, \boldsymbol\Omega; \mathbb P\1_{\Delta^{(-)}}).
\end{align*}
Now relations \eqref{eq:isom} and \eqref{deltapm:eq} 
follow from Theorem~\ref{XY1} applied separately to $\Delta^{(+)}$ and 
to $\Delta^{(-)}$.  
\end{proof}

The next assertion shows that if $\Delta={\mathbb T}$, then the wave operators $W_\pm (T (\omega), \boldsymbol\Omega; \mathbb P)$
are  complete.

\begin{theorem}\label{compl}
If  $\omega=\bar{\omega}\in  \mathbb W^{2, 1}_{\rm loc}({\mathbb T})$, then all conclusions of Theorem~\ref{deltapm:thm}  are true and
\begin{align}\label{deltacompl:eq}
\Ran W_\pm (T (\omega), \boldsymbol\Omega; \mathbb P)
= \mathbb H^2
\end{align}
or, equivalently,
 \[
  W_\pm (T (\omega), \boldsymbol\Omega;  \mathbb P ) W_\pm (T (\omega), \boldsymbol\Omega;  \mathbb P )^*=I
 \]
 $($the identity operator on $ {\mathbb H}^2)$.
In this case the operators $T (\omega)$ and the restriction of $\boldsymbol\Omega$ 
 to  $\Ran \1_{\Delta^{(\pm)}}$
are unitarily equivalent.
 \end{theorem}

\begin{proof}
The completeness \eqref{deltacompl:eq} is equivalent to the first equality \eqref{WOcompl:eq}.
\end{proof}

Let us state two consequences of this result. The first one concerns the spectral multiplicity of the Toeplitz operator $T(\omega)$.

\begin{corollary}\label{compl1}
Let
an interval $\Lambda$   satisfy condition \eqref{eq:TH}, and
let the numbers $n^{(\pm)}= n^{(\pm)}(\Lambda) $   be defined by formula \e{eq:AU}.
Then $n^{(+)} = n^{(-)}$ and 
the spectral representation of  
the operator $ T(\omega) E_{T(\omega) } (\Lambda)$ is realized on the space 
$L^{2}(\Lambda; {\mathbb C}^{n^{(\pm)}})$.
In other words, the spectral multiplicity 
of the operator $T(\omega)$ on $\Lambda$ equals $n^{(+)} = n^{(-)}$.
 \end{corollary}
 
 \begin{proof}
 According to Theorem \ref{omega} the spectral multiplicity 
 of the operator $\boldsymbol\Omega \1_{\Delta^{(\pm)}}$ on $\Lambda$ equals $ n^{(\pm)}$.  
 It follows from Theorem~\ref{compl} that the same statement is true
for the operator $T(\omega)$. This automatically implies that $n^{(+)} = n^{(-)}$.
  \end{proof}
  
The second corollary is a direct consequence of the asymptotic completeness \e{deltacompl:eq}.
  
\begin{corollary}\label{compl2}
For every $f\in{\mathbb H}^2$ asymptotic relation \e{evo:eq} $($with  $T= T(\omega))$  is satisfied where necessarily 
\[
\1_{\Delta^{(\pm)}}f_{\pm}=  W_\pm (T (\omega), \boldsymbol\Omega; \mathbb P)^* f.
\]
 For example,
we can set $f_{\pm}=  W_\pm (T (\omega), \boldsymbol\Omega; \mathbb P)^*f$.
\end{corollary}

Theorem~\ref{compl} and its corollaries conclude our construction of 
scattering theory for Toeplitz operators $T(\omega)$ 
with smooth symbols $\omega$. Since $\sigma (T(\omega)) =\sigma (\boldsymbol\Omega)$,
 the spectra of 
such Toeplitz  operators are thick (see Definition~\ref{class}). 
Note also that the scattering theory approach presented 
here gives (see Corollary \ref{compl1})   
an independent proof of Theorem \ref{Ros_smooth:thm} for smooth symbols.

If the symbol $\omega$ has jump discontinuities, then the equality \eqref{deltacompl:eq} is no longer true. 
To ensure the asymptotic completeness in this case, we need to take into account the 
 wave operators produced by   jumps of $\omega$. Such wave operators will be constructed in the next section.

 \section{Jump discontinuities. A model operator}\label{jumps:sect}
 
    
 \subsection{A model singularity}\label{model:subsect}
 As  a model operator, we choose the Toeplitz operator 
 $T_{\rm s}=T(\1_{(\z_{1}, \z_{2})})$ whose symbol $\1_{(\z_{1}, \z_{2})}$ 
 has two jumps at the points $\z_{1}$ and $\z_{2}$. 
The unitary operator $\Phi_{s}=\Phi(\1_{(\z_{1}, \z_{2})}): {\mathbb H}^{2}\to L^2 (0,1)$, 
diagonalizing $T_{\rm s}$, 
is given by formula \e{eq:Ros8} with 
$\varphi(z; \lambda) = \varphi_{s}(z;\lambda), z\in \mathbb D,$ 
defined in \e{eq:RS}. 
We suppose that $ \tilde{f} = \Phi_{s}f \in C_{0}^\infty (0,1)$ so that also 
 $ \hat{f}:= \varkappa \tilde{f}\in C_{0}^\infty (0,1)$ 
 where $\varkappa$ is given by \e{eq:RS2}. 
The integral \eqref{eq:Ros7} takes the form
\begin{align} 
(e^{-iT_{\rm s} t} f)(z)=  
(1-z / \z_1)^{-1/2} 
(1-z/\z_{2})^{-1/2}F(z, t)
\label{eq:F1}\end{align} 
where
\begin{align} 
F(z, t) =  \int\limits_0^1
(1-z / \z_1)^{-i\sigma (\lambda)} 
(1-z/\z_{2})^{i\sigma (\lambda)}
\hat {f}(\lambda) 
 e^{-i\lambda t} d\lambda
\label{sing:eq}\end{align} 
with the real-valued $\sigma(\lambda)$ defined by \eqref{eq:RS1}.
 
We  are interested in the behaviour of the function \e{eq:F1} as $|t|\to\infty$.  It is natural to expect that neighborhoods of the points $\z_{1}$ and $\z_{2}$ give the main contributions to the asymptotics of \e{eq:F1}. To see this, we have to estimate the integral in \eqref{sing:eq}. 
Note 
that $| \arg (1-z/\z_{j})|\leq\pi/2$ and hence
\begin{equation}
 \big|  (1-z/\z_{j})^{i (\pm  1)^j \sigma(\lambda)} \big| 
 = \exp \big(  {(\mp  1)^j \sigma(\lambda)} \arg(1-z/\z_{j})\big)
 \leq e^{\pi |\sigma (\lambda)| /2}.
\label{eq:arg}\end{equation}

  \begin{lemma} \label{RS1}  
  Suppose that $\tilde f\in C^\infty_0(0, 1)$. Then 
for all $p\geq 0$, we have the estimate
\begin{align}
|(e^{-iT_{\rm s} t} f)(z)| \leq & C_p |1-z / \z_1|^{-1/2 } |1-z / \z_{2}|^{-1/2 }
\nonumber\\
& \times \bigg(1+ \big|\ln |1-z / \z_1|\big| + \big|\ln |1-z / \z_2|\big|\bigg)^p
|t|^{-p}, 
\label{eq:FF}\end{align} 
with a constant $C_{p}$ independent of $t\not =0$ and $z\in \mathbb D$. 
In particular, if $|z-\z_1|\ge c>0$ and $|z-\z_2|\ge c>0$, then 
\begin{align*} 
|(e^{-iT_{\rm s} t} f)(z)| \leq  C_p |t|^{-p},\q \forall p\geq 0.
\end{align*}
  \end{lemma}
  
\begin{proof} 
Using that
$\hat f  = \varkappa \tilde f \in C^\infty_0(0, 1)$ and integrating in \eqref{sing:eq} 
by parts $p$ times,  we get
\begin{align*}
F(z, t) = (it)^{-p} \int\limits_0^1   
e^{-i\lambda t} 
\frac{d^p}{d\lambda^p} 
\bigg((1-z / \z_1)^{-i\sigma (\lambda)} 
(1-z/\z_{2})^{i\sigma (\lambda)}
\hat {f}(\lambda)\bigg) d\lambda,
\end{align*}
Taking into account \eqref{eq:arg}, we see that the derivative under the integral sign does not exceed 
\begin{align*}
\tilde C_p\bigg(1 +  \big|\ln |1-z / \z_1|\big| 
+ \big|\ln |1-z / \z_2|\big|\bigg)^p.
\end{align*}
In view of \e{eq:F1},
this leads to the required bound. 
\end{proof}

This lemma shows that  $(e^{-iT_s t} f)(z)$ ``lives" in neighborhoods of the points $\z_{1}$ and $\z_2$ as $|t|\to\infty$. This result 
can be made more precise if one takes into account the  dependence on the sign of $t$ as $t\to\pm\infty$. 

\begin{lemma} 
Let  $j=1$, $k=2$ for $t>0$ and  $j=2$, $k=1$  for $t< 0$. 
Assume that $|z-\z_j| \leq c_{j}<1$ and $|z-\z_k| \geq c_k >0$. Then for 
all   $p\geq 0$, we have the estimates
\begin{equation}
|(e^{-iT_{\rm s} t} f)(z)| \leq C_p |1-z / \z_j |^{-1/2 }  \big(   \big|\ln |1-z / \z_j|\big|+ 
|t| \big)^{-p},
\label{eq:RS6}
\end{equation}
with a constant $C_p$ independent of $t\not = 0$ and $z\in \mathbb D$.
\end{lemma}

   \begin{proof} 
Consider, for example, the case $j=1$,  $k=2$, $t>0$. 
The integral \eqref{sing:eq} can be rewritten as 
\begin{align*}
F(z, t) = 
\int_0^1 e^{-i\lambda t
-i \sigma (\lambda) \ln  (1-z / \z_1)} (1-z / \z_{2})^{ i \sigma (\lambda)}\hat{f}(\lambda) d\lambda .
\end{align*}
Integrating   by parts once, we see that
 \begin{align}
 F(z, t) = i  \int_0^1   e^{-i\lambda t
-i \sigma (\lambda) \ln  (1-z / \z_1)} 
\left( 
\frac{(1-z / \z_{2})^{ i \sigma (\lambda)}\hat{f}(\lambda)}
{t + \sigma' (\lambda) \ln  (1-z / \z_1)} 
 \right)' d\lambda,
\label{eq:RS8}
\end{align}
where $\sigma' (\lambda)= - (2\pi \lambda (1-\lambda))^{-1}
\leq  -2\pi^{-1}$. 
Since $|1-\z_1/z| \leq c_{1}<1$, we have 
\begin{align*}
\Re  \ln  (1-z / \z_1) =\ln  |1-z / \z_1| \leq  \ln c_1 <0,
\end{align*}
so that
\begin{align*}
|  t+ \sigma' (\lambda) \ln  (1-z / \z_1) | \geq c (-\ln  |1-z / \z_1| + t) .
\end{align*}
We also have the estimate $|\ln (1- z/\z_2)|\ge c_{2}>0$. 
  Using also 
\eqref{eq:arg} we conclude that 
the integral \e{eq:RS8} is bounded 
by $C(-\ln  |1-z / \z_1| + t)^{-1}$. Further integrations by parts show that
\[
|F(z, t)|\leq  C_p (-\ln  |1-z / \z_1| + t)^{-p}
\]
for all $p\geq 0$.
Substituting this estimate into \e{eq:F1}, we get \e{eq:RS6} for $j=1$.
\end{proof} 
        
\begin{corollary} \label{RS2c} 
Let  $j=1$  for $t>0$ and  $j=2$  for $t< 0$. Then
\begin{equation*}
\int_{(\z_{j}e^{-i\varepsilon}, 
\z_{j}e^{i\varepsilon})}|(e^{-iT_{\rm s} t} f)(\z)|^2 d{\bf m} (\z) \leq C_p  |t|^{-p}
\end{equation*}
 for 
all   $p\geq 0$ if $\varepsilon$ is sufficiently small.
\end{corollary}
        
 Thus   the function $(e^{-iT_{\rm s} t} f)(z)$  
tends to concentrate near $\z_{2}$ (resp. $\z_1$) as $t\to +\infty$ (resp.
  $t\to -\infty$).

\subsection{Jump discontinuities}\label{jump:subsect}


Here we consider symbols $\omega (\z)$ with a jump 
discontinuity at some point $\z_{0}\in {\mathbb T}$. We suppose that both one-sided limits 
\[
\omega (\z_{0}\pm 0): =  \lim_{  \varepsilon\to\pm 0}  \omega (\z_{0}e^{i  \varepsilon})=\alpha_{\pm}
\]
exist,  $\alpha_+\not = \alpha_-$, and
\begin{equation}
 \omega (\z)- \alpha_{\pm} = O (| \ln |\z-\z_{0}||^{-\rho}) \q \text{ for some } \q \rho>3/2
\label{eq:Ju}\end{equation}
  as $\z\to \z_{0}\pm 0$. 
  Pick some $\varepsilon\in (0,2\pi)$ and   define 
  the symbol $\omega_0 = \omega_0^{(\varepsilon)}$ by
\begin{equation}
\left.
\begin{aligned}
\omega_0(\z)& = \alpha_{-} \q {\rm for}\q \z\in (e^{-i\varepsilon}\z_{0},\z_{0})
\\
\omega_0(\z)& = \alpha_{+}\q {\rm for}\q \z\not\in (e^{-i\varepsilon}\z_{0},\z_{0})   
\end{aligned}
\right\} \q{\rm if }\q \alpha_{+}< \alpha_{-},
\label{eq:Ju+}
\end{equation}
and
\begin{equation}
\left.
\begin{aligned}
\omega_0(\z)&=\alpha_{+} \q {\rm for}\q \z\in (\z_{0}, e^{i\varepsilon}\z_{0})
\\
\omega_0(\z)&= \alpha_{-}\q {\rm for}\q \z\not\in(\z_{0}, e^{i\varepsilon}\z_{0})    
\end{aligned}
\right\} \q{\rm if }\q \alpha_{+}> \alpha_{-}.
\label{eq:Ju-}
\end{equation}
Now we apply the results of Section~\ref{model:subsect} to the unitary 
group $\exp(-i T(\omega_0)t)$. 
Note that 
\begin{align}\label{prop+:eq}
e^{-i T(\omega_0) t}= e^{-i \alpha_+ t} 
\exp\Big(-i (\alpha_{-}- \alpha_{+}) T (\1_{(e^{-i\varepsilon}\z_{0},\z_{0})} )t\Big)
\q {\rm if} \q \alpha_+< \alpha_-, 
\end{align}
and 
\begin{align}\label{prop-:eq}
  e^{-i T(\omega_0) t}= e^{-i \alpha_- t} 
  \exp\Big(-i (\alpha_+ - \alpha_-) T (\1_{(\z_{0},e^{i\varepsilon}\z_{0})} )t \Big) 
  \q {\rm if} \q \alpha_+> \alpha_-.
\end{align}

\begin{theorem} \label{JUMP} 
Let $\omega =\bar{\omega}\in L^\infty ({\mathbb T})$, 
and let condition \e{eq:Ju} be satisfied for some numbers $\alpha_{\pm}\in{\mathbb R}$. 
Define the symbol  $\omega_0$  by the formula \e{eq:Ju+} or \e{eq:Ju-}. 
If $\alpha_\pm < \alpha_\mp$, then the wave operator $W_{\pm} (T(\omega), T(\omega_0))$ exists. 
\end{theorem}

\begin{proof} 
Consider, for example, the wave operator $W_+ (T(\omega), T(\omega_0))$ 
for the case $\alpha_+ < \alpha_-$ and set $\omega_{\rm s} = \1_{(e^{-i\varepsilon}\z_0, \z_0)}$. 
According to Theorem~\ref{Cook} it suffices to check that 
\begin{equation}
\int_{1}^{\infty} \| (\omega - \omega_0)e^{-i T(\omega_{\rm s})  t} f\| dt <\infty
\label{eq:JuC}
\end{equation}
for functions $f$ such that $\tilde{f}= \Phi( \omega_{\rm s})f 
\in C_{0}^\infty ( 0, 1)$. 
Let $\delta (\z_{0})\subset{\mathbb T}$, $\delta (  \z_0, \varepsilon)\subset{\mathbb T}$ 
be disjoint neighbourhoods of the points 
$\z_{0}$, $e^{-i\varepsilon}\z_{0}$ and
\[
\Sigma (\z_{0}, \varepsilon) ={\mathbb T}\setminus \big( \delta (\z_{0})\cup \delta (\z_{0},\varepsilon)\big).
\]
Lemma~\ref{RS1} 
  and Corollary~\ref{RS2c} 
show that
 \begin{equation}
\|e^{-i T( \omega_{\rm s})  t} f\|_{L^2 (\Sigma(\z_{0}, \varepsilon))} 
 + 
\|e^{-i T( \omega_{\rm s})t} f\|_{L^2 ( \delta (\z_0, \varepsilon))}
\leq C_p t^{-p},\q \forall p\geq 0.
\label{eq:FFa}\end{equation}
In a neighborhood of the point $ \z_{0}$, we use condition  \e{eq:Ju}  so that
 \[
 \omega (\z)-  \omega_{0} (\z) = O (| \ln |\z-\z_{0}||^{- \rho}),   \q \rho>3/2,
\]
and hence in view of \e{eq:FF}  
 \[
|(\omega(\z)-\omega_0(\z))(e^{-i T( \omega_{\rm s})  t} f)(\z)| \leq 
C_p |1-\z / \z_0|^{-1/2 }  \big|\ln |1-\z / \z_0|\big|^{- \rho+p} t^{-p}.
\]
 If $p < \rho - 1/2$, then the right-hand side is in $L^2(\delta(\z_0))$, and it follows that
 \begin{equation}
\| (\omega-\omega_{0}) e^{-i T( \omega_{\rm s})  t} f\|_{L^2 (\delta ( \z_{0}))}
\leq C_p t^{-p}. 
\label{eq:FFc}\end{equation}
Since $\rho >3/2$, we can choose $p>1$.
 Combining  estimates \e{eq:FFa}
 and \e{eq:FFc}, we see that the integral \e{eq:JuC} converges.
 \end{proof}


\begin{remark} \label{JUMPe} 
The symbol $\omega_0 = \omega_0^{(\varepsilon)}$ depends on $\varepsilon\in (0, 2\pi)$, and hence so does  
the wave operator 
$W_\pm  (T(\omega), T (\omega_0^{(\varepsilon)}))$.   
However this dependence is trivial. Indeed, for two different positive 
$\varepsilon$ and $\nu$, we have
\begin{align*}
   e^{iT t}e^{- iT^{(\nu)}_0 t} 
   = \big(  e^{iT t}e^{- iT^{(\varepsilon)}_0  t} \big)  
   \big(  e^{iT^{(\varepsilon)}_0  t}  e^{- iT^{(\nu)}_0 t}\big)
\end{align*}
where we have set $T = T(\omega)$, $T^{(\varepsilon)}_0 = ~T (\omega_0^{(\varepsilon)})$, for short.
   By Theorem~\ref{JUMP}, the left-hand side and the first factor 
   on the right  converge, as $t\to \pm\infty$, to 
 $W_\pm(T, T^{(\nu)}_0)$ and 
 $W_\pm(T, T^{(\varepsilon)}_0)$, respectively. 
 Again by Theorem~\ref{JUMP}, the second factor on the right  converges to
the  operator 
\begin{equation*}
 W_\pm (T^{(\varepsilon)}_0, T^{(\nu)}_0) =: U_\pm  ( \varepsilon, \nu),
\end{equation*}
 so that
\begin{equation}
W_\pm(T, T^{(\nu)}_0) = W_\pm(T, T^{(\varepsilon)}_0) U_\pm(\varepsilon, \nu).
\label{eq:Uu}
\end{equation}
Since $U_\pm(\varepsilon, \nu)^* =  U_\pm(\nu, \varepsilon)$ and  $U_\pm(\varepsilon, \nu)  U_\pm(\nu, \varepsilon)=I$, 
the   operators $U_\pm(\varepsilon, \nu)$  
  are  unitary. Note that, by definitions \eqref{prop+:eq} and \eqref{prop-:eq}, 
they do not depend on $\alpha_+$ or $\alpha_-$.
 \end{remark}


Although we omit the dependence of the symbol $\omega_0$ 
on $\varepsilon$, we always keep in mind the relation \e{eq:Uu}.

\subsection{Orthogonality of the channels} 

Let us show that the ranges of the wave operators constructed in Theorems~\ref{deltapm:thm} and \ref{JUMP} are orthogonal to each other.
Consider first the wave operators corresponding to jumps of the symbol. Recall that, for an arc $(\z_{1}, \z_{2})\subset {\mathbb T}$,  the unitary operator $\Phi ( \1_{ (\z_{1}, \z_{2})}): {\mathbb H}^{2}\to L^2 (0,1)$ is defined by formulas   \e{eq:Ros3}, \e{eq:RS}.


\begin{theorem} \label{Orind}
   Let $(\z_{1}, \z_{2})\subset {\mathbb T}$, $(\z_{1}', \z_{2}')\subset {\mathbb T}$ where
     $\z_j \neq \z_k'$ for all $j,k=1,2$.  Suppose that, for some $\alpha, \alpha ' \in  {\mathbb R}$, the wave operators
   $W_\pm(T(\omega), T(\alpha\1_{ (\z_{1}, \z_{2})}))$ and    $W_\pm(T(\omega), T(\alpha '\1_{ (\z_{1}', \z_{2}')}))$
exist.  Then the  subspaces
   \[
\Ran   W_\pm  (T(\omega), T( \alpha\1_{ (\z_{1}, \z_{2})}))   \q \text{and} \q \Ran  W_\pm(T(\omega), T( \alpha '\1_{ (\z_{1}', \z_{2}')}))
   \]
   of ${\mathbb H}^{2}$ are orthogonal to each other.
\end{theorem}

\begin{proof} 
Denote $T_{s} = T\big(\1_{ (\z_{1}, \z_{2})}\big)$, $T'_{s} = T\big(\1_{ (\z_{1}', \z_{2}')}\big)$.
It suffices to check that
\begin{equation}
\lim_{|t|\to \infty} (e^{-i \alpha T_{s}t} f,  e^{-i \alpha' T'_{s} t} f')=0
\label{eq:ori}
\end{equation}
for all $f, f'\in {\mathbb H}^{2}$ such that 
$\Phi (\1_{ (\z_{1}, \z_{2})}) f \in C_{0}^{\infty} (0,1)$ 
and $ \Phi (\1_{ (\z_{1}', \z_{2}')}) f' \in C_{0}^{\infty} (0,1)$. 
Let $\delta\subset {\mathbb T}$ be a neighborhood of the set $\{\z_{1}, \z_{2}  \}$ such that $\dist\{\z_{j}', \delta\}>0$ 
 for both $j=1,2$. 
 Similarly,   let $\delta'\subset {\mathbb T}$   be a neighborhood of the set 
 $\{\z_{1}', \z_{2}'  \}$ such that 
 $\dist\{\z_{j} , \delta'\}>0$ for both $j=1,2$. We split the integral in \e{eq:ori} over ${\mathbb T}$ in three integrals: over
$\delta$, $\delta'$ and  
$\Sigma = {\mathbb T}\setminus (\delta\cup \delta')$, According to Lemma \ref{RS1}
$\| e^{-i \alpha T_{s} t} f\|_{L^{2}( \delta')} \to 0$, 
$\| e^{-i \alpha'  T_{s}' t} f' \|_{L^{2}( \delta)} \to 0$ and both factors 
$  \| e^{-i \alpha T_{s} t} f\|_{L^{2}( \Sigma)} \to 0$, 
$\| e^{-i \alpha'  T_{s}' t} f' \|_{L^{2}( \Sigma)} \to 0$ as $|t|\to \infty$.
This implies \eqref{eq:ori}. 
\end{proof} 
      
Next, we compare the wave operators 
for the pair $\boldsymbol\Omega$, $T(\omega)$ with those for the pair 
$T(\alpha\1_{(\z_1, \z_2)}), T(\omega)$.   

%
      
       \begin{theorem} \label{Orth} 
          Let $(\z_{1}, \z_{2})\subset {\mathbb T}$ and $\alpha  \in  {\mathbb R}$.
Suppose that  the wave operators $W_\pm  (T(\omega), \boldsymbol\Omega ; {\mathbb P})$ and
$W_\pm(T(\omega), T(\alpha \1_{ (\z_{1}, \z_{2})}))$  
exist. 
Then the  subspaces
\begin{align*}
\Ran   W_\pm  (T(\omega), \boldsymbol\Omega ;{\mathbb P})   
\q \text{and} \q \Ran  W_\pm(T(\omega), T( \alpha \1_{ (\z_{1}, \z_{2})}))
\end{align*} 
of $\mathbb H^{2}$ are orthogonal to each other.
\end{theorem}

   \begin{proof}  
   It suffices to check that 
    \begin{equation*}
   \lim_{|t|\to \infty} ({\mathbb P}e^{-i \boldsymbol\Omega t} f, 
   e^{-i T_{s} t} f_{0})=0,\q T_{s} = T(\alpha \1_{ (\z_{1}, \z_{2})}), 
 \end{equation*}
 for all $f\in \mathcal H^{\rm(ac)}_{\boldsymbol\Omega}$ and all 
 $f_{0}\in {\mathbb H}^{2}$ such that 
 $\Phi  (\1_{ (\z_{1}, \z_{2})})f_{0}\in C_{0}^{\infty} (0,1)$.
 Let $\psi_{\epsilon}\in C^{\infty} ({\mathbb T})$ be a real-valued function such that 
 $\psi_{\epsilon}(\z)=1$ in $\epsilon$-neighborhoods 
 of the points $\z_{1}$ and $\z_{2}$ and $\psi_{\epsilon}(\z)=0$ 
 away from $2\epsilon$-neighborhoods of these points. Put 
  $\phi_{\epsilon}(\z)=1-\psi_{\epsilon}(\z)$. Clearly,
\[
|({\mathbb P}e^{-i \boldsymbol\Omega t} \psi_{\epsilon} f, 
 e^{-i T_{s} t} f_{0})|\leq \| \psi_{\epsilon} f\| \| f_{0}\|
\]
tends to $0$ as $\epsilon\to 0$  uniformly in $t$. 

Thus,  we only have to show that
\begin{equation*}
\lim_{|t|\to \infty} ({\mathbb P} 
e^{-i \boldsymbol\Omega t}\phi_{\epsilon} f,   e^{-i T_{s} t} f_{0})=0, \q {\mathbb P} ={\mathbb P}_{+},
 \end{equation*}
 for a fixed $\epsilon> 0$. Recall that according to
  \e{eq:comm} the commutator
   \[
  [\mathbb P, \phi_{\epsilon}]={\mathbb P}_{+}\phi_{\epsilon}-\phi_{\epsilon} {\mathbb P}_{+}
  ={\mathbb P}_{+}\phi_{\epsilon} {\mathbb P}_-  - {\mathbb P}_- \phi_{\epsilon} {\mathbb P}_{+}
  \] 
 is compact  (actually, it belongs to the trace class).
 Thus property (g), see Sect.~\ref{ST:subsect}, implies that
\begin{align*}
\lim_{|t|\to \infty}  \| \, [\mathbb P, \phi_{\epsilon}] e^{-i\boldsymbol\Omega t} f\|= 0,  
\end{align*} 
 and hence it suffices to check that
\begin{equation}
\lim_{|t|\to \infty} ({\mathbb P}e^{-i \boldsymbol\Omega t}  f, \phi_{\epsilon}  
e^{-i T_{s} t}f_{0})=0.
\label{eq:ort2}
\end{equation}
According to Lemma~\ref{RS1} we have the estimate
\[
\|\phi_{\epsilon}  e^{-i T_{s} t}  f_{0}\| \leq C_p |t|^{-p},\q \forall p >0,
\]
whence \e{eq:ort2} follows.
\end{proof}

\section{Putting things together. Main result}

Now we are in a position to develop scattering theory for
Toeplitz operators $T=T(\omega)$ with piecewise continuous symbols $\omega$. 
From now on, we always assume that 
Condition~\ref{SM:cond} is satisfied with some finite set 
 ${\sf S} = \{\eta_1, \eta_2, \dots\}$.

\subsection{The existence of wave operators} 

Here, we state several immediate consequences of the results established 
in Sections~\ref{mult:sect} and \ref{jumps:sect}.

The first theorem is a  special case of parts (i)  and (ii)
of Theorem~\ref{deltapm:thm}  
with the set $\Delta = \mathbb T\setminus {\sf S}$.

\begin{theorem} \label{INCL2}
 Suppose that Condition \ref{SM:cond} is satisfied 
 and that $\omega\in{\mathbb W}^{2,1}_{\rm loc}({\mathbb T}\setminus {\sf S})$. 
Let the open sets $\Delta^{(\pm)}\subset{\mathbb T}\setminus {\sf S}$ 
be as defined in \eqref{eq:dpm}. 
Then the wave operators  $W_{\pm}(T , \boldsymbol\Omega; {\mathbb P})$ 
exist and  are isometric on the subspaces $\Ran \1_{\Delta^{(\pm)}}$. 
Moreover, they satisfy relations \e{deltapm:eq}. 
\end{theorem}   

Next, we consider the wave operators produced by the discontinuities of $\omega$. 
Let ${\sf S}^{(\pm)}, {\sf S}_0,$ be the sets in the union \eqref{SS:eq}. Every 
singular point $\eta_k\in {\sf S}^{(\pm)}$ 
 (but not $\eta_k\in {\sf S}_0$) 
 produces a new channel of scattering. 
Following the construction in Section~\ref{jump:subsect}, we choose an arbitrary 
$\varepsilon  \in (0,2\pi)$ and 
introduce auxiliary symbols
\renewcommand{\theequation}{\arabic{section}.\arabic{equation}+}
\begin{equation}
 \left.
\begin{aligned}
\omega_{k}(\z)&= \omega (\eta_k - 0)\q {\rm for}\q 
\z\in (e^{-i \varepsilon} \eta_k, \eta_k) 
\\
\omega_{k}(\z)&= \omega (\eta_k + 0)\q {\rm for}\q \z\notin (e^{-i\varepsilon} \eta_k, \eta_k)  
\end{aligned}
\right\} \q{\rm if }\q  \eta_k\in {\sf S}^{(+)}
\label{eq:G+}
 \end{equation}
\setcounter{equation}{0}
\renewcommand{\theequation}{\arabic{section}.\arabic{equation}--}
and
  \begin{equation}
 \left.
\begin{aligned}
\omega_{k} (\z)&= \omega (\eta_k + 0)\q 
{\rm for}\q \z\in (\eta_k, e^{i \varepsilon} \eta_k)
\\
\omega_{k}(\z)&= \omega (\eta_k - 0)\q {\rm for}\q \z\not\in 
(\eta_k, e^{i\varepsilon}\eta_k) 
\end{aligned}
\right\} \q{\rm if }\q \eta_k\in {\sf S}^{(-)}.
 \label{eq:G-}\end{equation} 
\renewcommand{\theequation}{\arabic{section}.\arabic{equation}}
The Toeplitz operator $T_{k}=T(\omega_{k})$ has 
simple absolutely continuous spectrum that coincides with the interval $\Lambda_k$ defined in \eqref{Lk:eq}.

Theorem~\ref{JUMP}  implies the following result.
  
\begin{theorem} \label{INCL3} 
Suppose that the condition \e{eq:Ju} holds at some point $\eta_{k}\in {\sf S}^{(\pm)}$. 
Let the symbol $\omega_{k}$  be defined by relations \e{eq:G+} or \e{eq:G-}.
Then the wave operator $W_{\pm} (T , T_{k})$ 
exists and is isometric on ${\mathbb H}^2$. 
\end{theorem}
 
\begin{corollary}\label{INCL31}
The spectral representation 
of the operator $T$ restricted to 
the subspace ${\rm Ran} \ W_\pm(T , T _k) $ is realized on the 
space $L^2(\Lambda_k)$.
       \end{corollary} 
       
       Although symbols \e{eq:G+} and \e{eq:G-} depend on $\varepsilon $, Remark~\ref{JUMPe} shows that the dependence of $W_\pm (T, T _{k}  )$ 
 on $\varepsilon $ is trivial.

The next result follows from
Theorems~\ref{Orind} and \ref{Orth}. It shows that different   scattering channels
are orthogonal   to each other.
  
\begin{theorem} \label{INCL4} 
  Let the assumptions of Theorem~$\ref{INCL2} $ hold. 
  Suppose also that 
the condition \eqref{eq:Ju} is satisfied at all points $\eta_k\in {\sf S}^{(\pm)}$. 
Then the ranges $\Ran W_{\pm}(T , T _{k})$ 
for all  $\eta_k\in {\sf S}^{(\pm)}$ 
are orthogonal to  $\Ran W_\pm(T, \boldsymbol\Omega;{\mathbb P})$ and are pairwise orthogonal 
to each other.  
\end{theorem}

It follows from the above results that 
 \begin{equation}
\Ran W_\pm (T , \boldsymbol\Omega; {\mathbb P}   )
\oplus \ \bigoplus_{ \eta_k\in {\sf S}^{(\pm)}}\Ran W_{\pm} (T , T _{k} )\subset {\mathbb H}^2.
\label{eq:AC+}
\end{equation} 
Our final objective is to 
establish the equality in \eqref{eq:AC+}. In other words, 
we intend to prove the asymptotic completeness of the wave operators involved.

\subsection{Counting multiplicities}
    
  Our proof of the asymptotic completeness requires an elementary result of a general nature (cf. \cite[Theorem~1.5.7]{Yafaev1}) 
    concerning self-adjoint operators with finite spectral multiplicity.

\begin{theorem}\label{mult:lem}
Let $A$ be a self-adjoint operator on a 
Hilbert space $\mathcal{H}$, and let $\tilde A$ 
be its restriction to its invariant subspace $\widetilde{\mathcal{H}}$. 
Suppose that, for some interval $\Lambda\subset{\mathbb R}$, 
both operators $A$ and $\tilde A$, are unitarily equivalent to the 
 operator of multiplication by independent variable $\lambda$ on the space 
$L^{2}(\Lambda; {\mathbb C}^m)$ where $m < \infty$.
Then $\widetilde{\mathcal H} = \mathcal H$ and $\tilde A = A$.
\end{theorem}    
    
\begin{proof}
 We may assume that $\mathcal H = L^2(\Lambda; \mathbb C^m)$ and 
$(Af)(\lambda) = \lambda f(\lambda),\ f\in \mathcal H$, a.e. $\lambda\in\Lambda$. 
  Since $\tilde A$ is a restriction of   $A$, it acts as multiplication by $\lambda$ in the direct integral (see, for example, \S\S 7.1, 7.2 of the book \cite{BS})
\begin{align*}
  \int_\Lambda^{\bigoplus} G(\lambda) d\lambda   
\end{align*} 
where  $ G(\lambda)$ is a measurable family of some subspaces of $ \mathbb C^m$.
At the same time, $\tilde A$ is unitarily equivalent to multiplication by $\lambda$ 
on the space $L^2(\Lambda; \mathbb C^m)$. It follows  that 
$\dim G(\lambda) = m$ for a.e. $\lambda\in \Lambda$.  
Since $m < \infty$, we conclude that 
$G(\lambda) = \mathbb C^m$  and hence $\widetilde{\mathcal H} = \mathcal H$ and 
$\tilde A = A$, as required.
\end{proof}    
    
   \begin{corollary} \label{mult:cor} 
Let $A$ be a self-adjoint operator on a 
Hilbert space $\mathcal{H}$, 
and let $A_{k}$, $k = 0,\ldots, N$, $N<\infty$, 
be its restrictions to   pairwise orthogonal invariant subspaces 
$\mathcal{H}_{k}$. 
Suppose that, for some interval $\Lambda\subset{\mathbb R}$, 
the operators $A$ and $A_{k}$, $k = 0,\ldots, N$, 
are unitarily equivalent to the 
operators of multiplication by 
independent variable $\lambda$ in the spaces 
$L^{2}(\Lambda; {\mathbb C}^m)$, $m < \infty$, and
$L^{2} (\Lambda; {\mathbb C}^{m_{k}})$, $k = 0,\ldots, N$, respectively. 
Assume that
\begin{equation*}
\sum_{k=0}^{N} m_k = m.
\end{equation*}
Then
\[
\bigoplus_{k=0}^{N}
\mathcal{  H}_{k} = 
\mathcal{H} \q {\rm and} \q\bigoplus_{k=0}^{N}A_{k} = A.
\]
  \end{corollary}
  
  \begin{proof}
  It suffices to apply Theorem~\ref{mult:lem} to 
  the subspace $\widetilde{\mathcal H} = \oplus_{k=0}^{N}  
\mathcal{  H}_{k} \subset {\mathcal H} $ and the operator $\tilde A =  \oplus_{k=0}^{N}A_{k}$.
  \end{proof}
       
\subsection{Asymptotic completeness}

  In order to use Corollary \ref{mult:cor} for the proof of the 
  asymptotic completeness, 
we need to find spectral multiplicities 
for the operator $T$ restricted to the subspaces 
on the left-hand side of \eqref{eq:AC+}. 
The required result for $\Ran W_\pm(T, T_k)$ is given by Corollary 
\ref{INCL31}. 

Let us now consider   the first term in \eqref{eq:AC+}.
Below we systematically use the intertwining property \e{intertwine:eq}. 

 \begin{lemma} \label{iso} 
 Suppose that 
 the conditions of Theorem~\ref{INCL2} hold, and that an interval $\Lambda$ satisfies \eqref{eq:TH}. 
 Then the spectral representation of the operator $T $ restricted 
 to   the subspace
\begin{equation}
 \mathcal H^{(\pm)} (\Lambda) 
 : =\Ran \big( E_{T}(\Lambda) W_\pm (T , \boldsymbol\Omega; {\mathbb P} \1_{\Delta^{(\pm)}} )\big)
 \label{eq:ranloc}
\end{equation} 
 is realized on the space $L^{2}\big(\Lambda; {\mathbb C}^{n^{(\pm)}}\big)$.
     \end{lemma}

  \begin{proof}  
Let $\delta^{(\pm)} = \delta^{(\pm)}(\Lambda)$ be defined by \eqref{eq:AUU}. 
By Theorem \ref{omega}, the spectral representation of 
the operator $\boldsymbol\Omega$  
restricted to $L^2(\delta^{(\pm)})$ is realized 
on the space $L^2(\Lambda; \mathbb C^{n^{(\pm)}})$. According to Theorem~\ref{INCL2} 
 and the 
intertwining property \eqref{intertwine:eq}, the operator
\[
E_{T}(\Lambda) W_\pm (T , \boldsymbol\Omega; {\mathbb P} \1_{\Delta^{(\pm)}})
=  W_\pm (T , \boldsymbol\Omega; {\mathbb P} \1_{\delta^{(\pm)}}):  
L^2(\delta^{(\pm)})
\to  \mathcal H^{(\pm)} (\Lambda),
\]
is unitary. Therefore the 
spectral multiplicities of the operators 
$\left.\boldsymbol\Omega\right|_{L^2(\delta^{(\pm)})}$ and 
$\left.T\right|_{\mathcal{H}^{(\pm)}(\Lambda)}$ 
on the interval $\Lambda$ coincide, 
and both equal $n^{(\pm)}$. 
%
%
  \end{proof} 
 
First, we verify a ``local" form of the asymptotic completeness. 
Recall that the intervals $\Lambda_{k}$ are defined by \e{Lk:eq}, 
the subset ${\sf S}(\Lambda)\subset {\sf S}$ is distinguished by 
the condition $\Lambda\subset \Lambda_k$ for $\eta_{k}\in {\sf S}(\Lambda)$, 
and ${\sf S}^{(\pm)}(\Lambda) = {\sf S}(\Lambda)\cap {\sf S}^{(\pm)}$.
 
\begin{theorem} \label{MainL} 
Suppose that the symbol $\omega$   satisfies Condition 
\ref{SM:cond} with some finite set ${\sf S}$, and that 
$\omega \in  {\mathbb W}^{2,1}_{\rm loc}({\mathbb T}\setminus {\sf S})$. 
Let $\Lambda$ be an interval satisfying \eqref{eq:TH}, 
and assume that condition \e{eq:Ju} holds at each point 
$\eta_k\in {\sf S}^{(\pm)} ( \Lambda)$. Then 
 the equality
 \begin{align}
\Ran \big( E_{T}(\Lambda) W_\pm (T, & \boldsymbol\Omega; {\mathbb P}\1_{\Delta^{(\pm)}} )\big)\notag\\[0.2cm]
&\ \oplus  \bigoplus_{ \eta_k\in {\sf S}^{(\pm)}(\Lambda)}
\Ran \big( E_{T}(\Lambda)W_{\pm} (T, T_{k}) \big) = E_T(\Lambda){\mathbb H}^2
\label{eq:ACL+}
\end{align}
holds for both signs $``+ $'' and $``- $''. 
\end{theorem}
     
\begin{proof}
Note that the orthogonal sum over $\eta_k$ on the left-hand side of 
\eqref{eq:ACL+} contains 
$s^{(\pm)} =\# ({\sf S}^{ (\pm)}(\Lambda))$ terms. Thus, using the notation 
\begin{align*}
 \mathcal H_k^{(\pm)}(\Lambda) : = \Ran\big(E_T (\Lambda) W_\pm(T, T_{k})\big),
 \quad \eta_k\in {\sf S}^{(\pm)}(\Lambda),
\end{align*}
$\mathcal H (\Lambda) := E_T(\Lambda)\mathbb H^2$, 
and definition \eqref{eq:ranloc}, we can rewrite \eqref{eq:ACL+} as  
\begin{align}\label{compl:eq}
 \mathcal H^{(\pm)}(\Lambda)\oplus 
\bigoplus_{k=1}^{s^{(\pm)}}  \mathcal H_k^{(\pm)}(\Lambda) = \mathcal H (\Lambda).
\end{align}
In order to prove \eqref{compl:eq}, we make the following observations. 
 According to Lemma~\ref{iso} the 
 operator $T$ restricted to the subspace $\mathcal H^{(\pm)}(\Lambda)$
  has spectral multiplicity $n^{(\pm)}$.   
Similarly, Corollary \ref{INCL31}  
shows that, for every $k = 1, 2, \dots, s^{(\pm)}$ the operator $T$ on the subspace 
$\mathcal H_k^{(\pm)}(\Lambda)$ has spectral multiplicity $1$.

Now we use Corollary \ref{mult:cor} with 
 the spaces $\mathcal H (\Lambda)$, 
\begin{align*} 
 \widetilde{\mathcal H} (\Lambda) = \mathcal H^{(\pm)} (\Lambda)\oplus \bigoplus_{k=1}^{s^{(\pm)}}
  \mathcal H_k^{(\pm)} (\Lambda)\subset \mathcal H(\Lambda),
 \end{align*} 
 the operators 
\begin{align*}
A = \left.T\right|_{ \mathcal H(\Lambda)}, \q
 A_0 = \left.T\right|_{ \mathcal H^{(\pm)}(\Lambda)},\q 
 A_k = \left.T\right|_{  \mathcal H_k^{(\pm)}(\Lambda)},   
 \end{align*}
  and multiplicities
 \begin{align*}
m_0 = n^{(\pm)}, \q m_k = 1,\q k = 1, 2, \dots, s^{(\pm)}.
\end{align*}
According to Theorem \ref{Ros_smooth:thm}, 
the spectral multiplicity of the operator $A$  equals 
\begin{align*}
m = n^{(\pm)}+s^{(\pm)} =\sum_{k=0}^{s^{(\pm)}} m_k.
\end{align*}
Thus Corollary \ref{mult:cor} entails \eqref{compl:eq}, which completes the proof.
\end{proof}

The next theorem constitutes the main result of the paper. 

\begin{theorem} \label{Main} 
Suppose that the symbol $\omega$   satisfies Condition 
\ref{SM:cond} with some finite set ${\sf S}$, and that 
$\omega \in  {\mathbb W}^{2,1}_{\rm loc}({\mathbb T}\setminus {\sf S})$.
Assume that condition \e{eq:Ju} holds at each point 
$\eta_k\in {\sf S}^{(\pm)}$. Then 
\begin{enumerate}[\rm(i)]

\item 
The wave operators  $W_{\pm}(T(\omega), \boldsymbol\Omega; {\mathbb P})$ exist and satisfy relation \e{deltapm:eq}. These operators are isometric on the subspaces $\Ran \1_{\Delta^{(\pm)}}$ of $L^2 ({\mathbb T})$.  

\item
 Let  $\eta_k\in {\sf S}^{(\pm)}$, and let the symbols $\omega_{k}$ be defined by formulas \e{eq:G+}, \e{eq:G-}.  Then   the wave operators $W_{\pm} (T(\omega), T(\omega_{k}))$ 
exist and are isometric. 

\item
The ranges of all operators  $W_{\pm}(T(\omega), \boldsymbol\Omega; {\mathbb P})$ and $W_{\pm} (T(\omega), T(\omega_{k}))$ are orthogonal to each other.

\item 
 The asymptotic completeness holds:
\begin{equation}
\Ran W_\pm (T(\omega), \boldsymbol\Omega; {\mathbb P} \1_{\Delta^{(\pm)}}) 
\oplus  \bigoplus_{ \eta_k\in {\sf S}^{(\pm)}} \Ran W_{\pm} (T(\omega), T (\omega_{k}) ) 
= {\mathbb H}^2
\label{eq:AC++}\end{equation} 
for both  signs $``+ $'' and $``- $''. 
\end{enumerate}
     \end{theorem}
     
 \begin{proof}
 Assertions (i), (ii) and (iii) are direct consequences of Theorems~\ref{INCL2}, \ref{INCL3} and \ref{INCL4}, respectively.

Let us check (iv).
Since $\Lambda_{\rm exc}$ is closed,  
the set $G = (\gamma_1, \gamma_2)\setminus \Lambda_{\rm exc}$ is open. 
According to Theorem~\ref{MainL}
the relation \e{eq:ACL+} is true for every constituent open 
  subinterval of $G$. For every such   subinterval $\Lambda$, the orthogonal sum over $\eta_k\in {\sf S}^{(\pm)} (\Lambda)$ 
in \eqref{eq:ACL+} coincides 
with the sum over all $\eta_k\in {\sf S}^{(\pm)}$. 
Indeed, if $\eta_k\in{\sf S}^{(\pm)}\setminus {\sf S}^{(\pm)}(\Lambda)$, then  $E_{T_k}(\Lambda) = 0$ and hence,
by the intertwining property 
\eqref{intertwine:eq}, $E_{T}(\Lambda)W_{\pm} (T, T_{k}) = W_{\pm} (T, T_{k})E_{T_k}(\Lambda) = 0$. Therefore summing relations \e{eq:ACL+} over all constituent  
  subintervals of $G$, we obtain this relation
   for the set $G$ itself:
%
%
\begin{align}\label{glob:eq}
\Ran \big(E_{T}(G)W_\pm &( T, \boldsymbol\Omega; {\mathbb P} \1_{\Delta^{(\pm)}} )
\big)\notag\\
&\ \oplus \bigoplus_{ \eta_k\in {\sf S}^{(\pm)}} 
\Ran \big(E_{T}(G) W_{\pm} (T, T_k)    \big) = 
E_{T}(G){\mathbb H}^2,
\end{align}
where $T= T(\omega)$, $T_{k} = T(\omega_{k})$. 
Since the set $\Lambda_{\rm exc}$ has measure zero 
and the operator $T$ is absolutely continuous, 
we have $E_T(G) =   I$.  Thus \eqref{glob:eq} coincides with \eqref{eq:AC++}. 
\end{proof}

\subsection{Classification of the spectrum} 

Let us come back to the classification of the spectrum given by Definition~\ref{class}. 
 In this subsection we relate the subspaces $\mathcal {H}_{\rm thin}, \mathcal {H}_{\rm thick}$ 
defined in \eqref{decompose:eq} with the subspaces on the left-hand side of \eqref{eq:AC++}. 
We suppose that the conditions of Theorem~\ref{Main}  are satisfied, and define the subspaces
\begin{align} 
\begin{cases}
\mathcal H_{\rm thick}^{(\pm)} = &\ \Ran W_\pm (T , \boldsymbol\Omega; {\mathbb P}  ) ,\q T=T(\omega),
 \\[0.2cm]
\mathcal H_{\rm thin}^{(\pm)} = &\  \bigoplus_{ \eta_k\in {\sf S}^{(\pm)}} \Ran W_{\pm} (T , T _{k} ) ,\q
T_{k}=T(\omega_{k}),
\end{cases}
\label{eq:subsp}
\end{align}
of ${\mathbb H}^2$.

The  statement below (cf. Corollary~\ref{compl2}) is a direct consequence of the definition of the wave operators.

  \begin{lemma}\label{subsp}
For every $f\in \mathcal H_{\rm thick}^{(\pm)} $, asymptotic relation \e{evo:eq} holds with
  $f^{(\pm)}=  W_\pm (T , \boldsymbol\Omega; \mathbb P)^* f$. 
  For every $f\in \mathcal H_{\rm thin}^{(\pm)} $, asymptotic relation 
 \begin{align} 
  e^{-iTt}f \sim \sum_{ \eta_k\in {\sf S}^{(\pm)}} e^{-T_{k}t}f^{(\pm)}_{k},\q t\to\pm \infty,
 \label{eq:subsp1}\end{align}
     holds with
  $f_{k}^{(\pm)}=  W_\pm (T , T_{k})^* f$.
 \end{lemma}
 

Using Theorem~\ref{Main} it is easy to find a relation between the thick and thin subspaces defined by \e{decompose:eq} and  the subspaces \e{eq:subsp}.
 
  \begin{lemma}\label{subsp1}
  Under the assumptions of Theorem~\ref{Main}, we have
   \begin{equation} 
  \mathcal H_{\rm thick}  \subset  \mathcal H_{\rm thick}^{(+)} \cap   \mathcal H_{\rm thick}^{(-)}
  \label{eq:subsp2}  \end{equation} 
  and
    \begin{equation} 
  \mathcal H_{\rm thin} \subset  \mathcal H_{\rm thin}^{(+)} \cap   \mathcal H_{\rm thin}^{(-)}.
  \label{eq:subsp3} \end{equation} 
   \end{lemma}
   
    \begin{proof}
    First we check \e{eq:subsp2}.
    Let $ f\in  \mathcal H_{\rm thick}$.
    By definition \e{decompose:eq}, this means that
    $f=E_{T} (\sigma( \boldsymbol\Omega)\setminus  \Upsilon )f$. It now follows from formula \e{eq:AC++} that
 \begin{equation}
f=E_{T} (\sigma( \boldsymbol\Omega)\setminus  \Upsilon )W_\pm (T , \boldsymbol\Omega; {\mathbb P} ) f^{(\pm)}
+\sum_{ \eta_k\in {\sf S}^{(\pm)}}E_{T} (\sigma( \boldsymbol\Omega)\setminus \Upsilon )  W_{\pm} (T , T _{k})  f^{(\pm)}_{k}
\label{eq:subsp4}\end{equation} 
for some $f^{(\pm)}\in L^2 ({\mathbb T})$ and $f^{(\pm)}_{k}\in   {\mathbb H}^2$. In view of the intertwining property
\e{intertwine:eq} we can rewrite \e{eq:subsp4} as
 \begin{equation}
f= W_\pm (T , \boldsymbol\Omega; {\mathbb P} ) E_{\boldsymbol\Omega} (\sigma( \boldsymbol\Omega)\setminus  \Upsilon ) f^{(\pm)}
+\sum_{ \eta_k\in {\sf S}^{(\pm)}}   W_{\pm} (T , T _{k}) E_{T_{k}} (\sigma( \boldsymbol\Omega)\setminus  \Upsilon )  f^{(\pm)}_{k}.
\label{eq:subsp5}\end{equation} 
All terms in the sum 
 over $\eta_k$ vanish
because $\sigma (T_{k})\cap \big(\sigma( \boldsymbol\Omega)\setminus  \Upsilon \big)=\varnothing$.  Thus it follows from \e{eq:subsp5} that
\[
f= W_\pm (T , \boldsymbol\Omega; {\mathbb P} ) E_{\boldsymbol\Omega} (\sigma( \boldsymbol\Omega)\setminus  \Upsilon ) f^{(\pm)}
\in \mathcal H_{\rm thick}^{(\pm)}
\]
for both signs $``\pm"$.

  The inclusion  \e{eq:subsp3} is verified in a similar way.  Precisely, 
if $ f\in  \mathcal H_{\rm thin}$, then
$f=E_{T} (\sigma( T)\setminus \sigma( \boldsymbol\Omega))f$. 
Therefore using again  formula \e{eq:AC++} and
the intertwining property \e{intertwine:eq}, we find that 
\begin{align*}
f =  W_\pm (T , \boldsymbol\Omega; {\mathbb P} ) E_{\boldsymbol\Omega} 
(\sigma( T)&\setminus \sigma( \boldsymbol\Omega))f^{(\pm)}\notag\\
+ &\ \sum_{ \eta_k\in {\sf S}^{(\pm)}}   
W_{\pm} (T , T _{k}) E_{T_{k}} (\sigma( T)\setminus \sigma( \boldsymbol\Omega)) f^{(\pm)}_{k}.
\end{align*} 
Since the first term on the right is zero, we see that
    \[
f=  \sum_{ \eta_k\in {\sf S}^{(\pm)}}   W_{\pm} (T , T _{k}) E_{T_{k}} (\sigma( T)\setminus \sigma( \boldsymbol\Omega)) f^{(\pm)}_{k}\in \mathcal H_{\rm thin}^{(\pm)}
\]
 for both signs $``\pm"$.
     \end{proof}
     
     Combining Lemmas~\ref{subsp} and \ref{subsp1}, we find an asymtotic behavior of $e^{-iTt} f$ for $f\in  \mathcal H_{\rm thick} $ and  $f\in  \mathcal H_{\rm thin} $.
     
      \begin{theorem}\label{Subsp}
For every $f\in \mathcal H_{\rm thick}  $, asymptotic relations \e{evo:eq}    are satisfied  for both signs $``\pm"$ with
  $f^{(\pm)}=  W_\pm (T , \boldsymbol\Omega; \mathbb P)^* f$.
  For every $f\in \mathcal H_{\rm thin} $, asymptotic relations 
 \e{eq:subsp1}  are satisfied for both signs $``\pm"$ with
  $f_{k}^{(\pm)}=  W_\pm (T , T_{k})^* f$.
 \end{theorem}

For  $f\in \mathcal H_{\rm mix}  $, the  asymtotics of $e^{-iTt} f$ 
as $t\to\pm\infty$ may contain both terms ${\mathbb P}e^{-i\boldsymbol\Omega t} f^{(\pm)}$ 
and $ e^{-i T_{k}t} f_{k}^{(\pm)}$. 
This is illustrated  with the explicit example considered in 
the next subsection. It exhibits all three types of spectrum.

\subsection{Example} 
Consider the symbol shown in Fig. 1.
For convenience we copy this 
figure again with more detailed labelling, see Fig. 3. 
Below we use notation  \e{eq:Jump2} and \e{eq:dpm}.

Assume that $\omega'(\z)>0$ on the arcs $(\varkappa, \eta)$ and $(\eta, \nu)$, and 
$\omega'(\z) <0$ on the arc $(\nu, \varkappa)$. 
Thus the spectrum of $T(\omega)$ is simple and it coincides with the interval 
$[0, d]$. 
Also, 
\begin{align*}
\Delta^{(+)} = (\nu, \varkappa), \; \Delta^{(-)} = (\varkappa, \eta)\cup (\eta, \nu),\ \quad 
{\sf S}^{(+)} = \{\varkappa\}, \;{\sf S}^{(-)} = \{\eta\}  
\end{align*}  
and $\Lambda_{\varkappa}= [0,b]$, $\Lambda_{\eta}= [a,c]$
so that $\Upsilon=[0,c]$.
By Definition~\ref{class}, the thin, thick and mixed spectra 
coincide with the sets $[a, b]$, $[c, d]$ and $[0, a]\cup[b, c]$, 
respectively.

The model jump symbols are 
\begin{align*}
\omega_+(\z) = b\1_{(\varkappa e^{-i\varepsilon}, \varkappa)}(\z),\quad 
\omega_-(\z) = a+(c-a)\1_{(\eta, \eta e^{i\varepsilon}) }(\z),
\end{align*}
with a fixed $\varepsilon\in (0, 2\pi)$.  
It is clear that $\Lambda_{\rm exc} = \{0, a, b, c, d\}$. Thus the set 
$(0, d)\setminus\Lambda_{\rm exc}$ is the union of four intervals,
\[
(0, d)\setminus\Lambda_{\rm exc}=(a, b)\cup (c, d)\cup(0, a)\cup (b,c), 
\]
 each of which 
satisfies \eqref{eq:TH}. Consider them one by one. 
Below we use the notation \eqref{eq:AUU}: 
$\delta^{(\pm)} = \Delta^{(\pm)}\cap \omega^{-1}(\Lambda)$.

\begin{figure}
\resizebox{8cm}
{!}{\begin{tikzpicture}
\draw [very thick] plot[smooth, tension=.7] 
coordinates{(-4,-4) (-2.5,-3.5) (-1.5,-2)};
\draw [very thick] plot[smooth, tension=.7] coordinates {(-1.5,1.5) 
(-0.5,2.5) (0.5,2) (2,1.5) (3,0.5) (3.5,-0.5)};
\draw[dashed]
(-1.5,1.5) node (v2) {} -- (-1.5,-2) node (v3) {};
\draw[dashed]
(3.5,-0.5) node (v4) {} -- (3.5,-4) node (v13) {};
\draw [semithick][-latex](-5,-4) node (v1) {} -- (5,-4);
\draw [semithick][-latex](-5,-4) node (v12) {}  -- (-5,3.5);
\draw [semithick][dotted](-5,1.5) node (v8) {} -- (-1.5,1.5);
\draw [semithick][dotted](-5,-2) node (v5) {} -- (-1.5,-2);
\draw [semithick][dotted](2,1.5) node (v11) {} -- (-1.5,1.5);
\draw [semithick][dotted](3.5,-0.5) -- (-5,-0.5) node (v7){};
\draw [semithick][dotted](-0.5,2.5) node (v6) {} -- (-5,2.5) node (v9) {};
\draw [semithick][dotted] (-1.5,-2) node (v10) {} -- (3.5,-2);
\draw[dotted] (v6) -- (-0.5,-4) node (v15) {};
\draw[dashed] (v10) -- (-1.5,-4) node (v14) {};
\draw[dotted] (v11) -- (2,-4) node (v16) {};
\coordinate[label=left:$0$] (A) at (v12);
\coordinate[label=left:$a$] (A) at (v5);
\coordinate[label=left:$b$] (A) at (v7);
\coordinate[label=left:$c$] (A) at (v8);
\coordinate[label=left:$d$] (A) at (v9);
\coordinate[label=below:$\varkappa$] (A) at (-4, -4);
\coordinate[label=below:$\varkappa$] (A) at (v13);
\coordinate[label=below:$\eta$] (A) at (v14);
 \coordinate[label=below:$\nu$] (A) at (v15);
\coordinate[label=below:$\mu$] (A) at (v16); 
\end{tikzpicture}}\caption{Example}
\end{figure}

\underline{Thin spectrum.} 
Let $\Lambda = (a, b)$, 
so that 
\begin{align*}
{\delta}^{(+)}  = {\delta}^{(-)} = \varnothing,\quad n^{(+)}  = n^{(-)}  = 0, \q
s^{(+)} = s^{(-)}  = 1 . 
\end{align*}
 According to Theorem~\ref{Subsp}, for every 
$f\in E_T(\Lambda)\mathbb H^2 = \mathcal{H}_{\rm thin}$ we have 
\begin{align*}
e^{-iTt} f \sim \ e^{-i T(\omega_\pm) t} f^{(\pm)}, \q t\to\pm\infty,
\end{align*}
with $f^{(\pm)} = W_{\pm}(T, T(\omega_{\pm}))^* f$. This is consistent with \e{eq:subsp3}.

\underline{Thick spectrum. } 
Let $\Lambda = (c, d)$, so that 
\begin{align*}
{\delta}^{(+)} = (\nu, \mu), \;{\delta}^{(-)}  = (\eta, \nu),  \quad n^{(+)}  = n^{(-)}  = 1,
\quad
s^{(+)}  = s^{(-)}  = 0. 
\end{align*}

 According to Theorem~\ref{Subsp}, for every 
$f\in E_T(\Lambda)\mathbb H^2 = \mathcal{H}_{\rm thick}$ we have 
\begin{align*}
e^{-iTt} f \sim \ \mathbb P e^{-i \boldsymbol\Omega t} f^{(\pm)}, \q t\to\pm\infty,
\end{align*}
with $f^{(\pm)} = W_{\pm}(T, \boldsymbol\Omega; \mathbb P)^* f$. 
 This is consistent with \e{eq:subsp2}.

\underline{Mixed spectrum.} 
Let $\Lambda = (0, a)$, so that 
\begin{align*}
{\delta}^{(+)}  = \varnothing, \;
{\delta}^{(-)}  = (\varkappa, \eta),\quad n^{(+)}  = 0, \;
n^{(-)}  = 1\quad
s^{(+)} = 1, \; s^{(-)} = 0. 
\end{align*}
The asymptotic completeness \eqref{eq:ACL+} takes the form
\begin{align}\label{ex_mix1:eq}
\Ran W_+(T, T(\omega_+))E_{T(\omega_+)}(\Lambda) 
= 
\Ran W_-\big(T, \boldsymbol\Omega; \mathbb P\1_{(\varkappa, \eta)}\big)
= E_T(\Lambda) \mathbb H^2.
\end{align}
According to \eqref{eq:WO1x},   it follows from \eqref{ex_mix1:eq} that 
for every $f\in E_T(0, a)\mathbb H^2$ we have 
\begin{align}
e^{-iT t} f \sim e^{-i T(\omega_+)t} f^{(+)},\ t\to\infty\q \mbox{and}\q
e^{-iT t} f \sim \mathbb P e^{-i \boldsymbol\Omega t} f^{(-)},\ t\to-\infty,
\label{eq:mix}\end{align}
with  
\begin{align*}
f^{(+)}= W_+(T, &\ T(\omega_+))^* f\in E_{T(\omega_+)}(0, a) \mathbb H^2\\ &\ {\rm and}\q 
f^{(-)} = W_-\big(T, \boldsymbol\Omega; \mathbb P\big)^*f
\in E_{\boldsymbol\Omega}(0, a)L^2(\mathbb T). 
\end{align*}

Let $\Lambda = (b, c)$, so that 
\begin{align*}
{\delta}^{(+)} = (\mu, \varkappa),\; 
{\delta}^{(-)} = \varnothing, \quad n^{(+)}  = 1, \; n^{(-)}  = 0, 
\quad
s^{(+)} = 0, \; s^{(-)}  = 1 . 
\end{align*}
The asymptotic completeness \eqref{eq:ACL+} takes the form
\begin{align}\label{ex_mix2:eq}
\Ran W_+\big(T, \boldsymbol\Omega; \mathbb P\1_{(\mu, \varkappa)}\big)
=  
\Ran W_-(T, T(\omega_-))E_{T(\omega_-)}(\Lambda) = E_T(\Lambda)\mathbb H^2.
\end{align}
 It follows from \eqref{ex_mix2:eq} that 
for  $f\in E_T(b,c)\mathbb H^2$ the asymptotics of $e^{-iT t} f $ is given by relations similar to \e{eq:mix}.

Thus,  on the mixed spectrum, 
the operator $T$ has different evolution properties as $t\to\infty$ and $t\to-\infty$.

\end{document}